\documentclass[12pt]{article}
\usepackage{amsmath,amsfonts,amsthm}
\usepackage{bbm}    \usepackage[utf8]{inputenc}
\usepackage{cleveref}
\usepackage{tikz}
\usepackage{xr}
\newcommand\Fdist{P}
\newcommand\Fdistmax{Q}
\newcommand\Enabla{D}
\newcommand{\Herror}{H}
\newcommand{\Hprimeerror}{H'}
\newcommand{\supp}{\operatorname{supp}}
\newcommand{\Rndomain}{U}
\newcommand\weakto{\rightharpoonup}
\newcommand\Lin{\mathrm{Lin}}
\newcommand\Lip{\mathrm{Lip}}

\newcommand\vol{\mathrm{vol}}
\newcommand\sym{\mathrm{sym}}
\newcommand\Id{\mathrm{Id}}
\newcommand\id{\mathrm{id}}
\newcommand\R{\mathbb{R}}
\newcommand\N{\mathbb{N}}
\newcommand\eps{\varepsilon}
\newcommand{\dist}{\mathrm{dist}}

\numberwithin{equation}{section}
\newtheorem{theorem}{Theorem}[section]
\newtheorem{lemma}[theorem]{Lemma}
\newtheorem{definition}[theorem]{Definition}
\newtheorem{proposition}[theorem]{Proposition}
\newtheorem{corollary}[theorem]{Corollary}
\newtheorem{remark}[theorem]{Remark}

\newcommand\cof{\mathop{\mathrm{cof}}}
\newcommand\Cof{\mathop{\mathrm{Cof}}}
\newcommand\Det{\mathop{\mathrm{Det}}}

\DeclareMathOperator{\Div}{div}
\newcommand\Tr{\mathop{\mathrm{Tr}}}

\newcommand{\FPhi}{\hat\Phi}

\newcommand{\C}{\mathbb C}

\usepackage{fancyhdr}
\fancypagestyle{first}{%
\fancyhf{} \fancyfoot[C]{\thepage}
}
\newcommand\loc{\mathrm{loc}}
\begin{document}
\thispagestyle{first}
\begin{center}
{ \Large
Uniform $W^{2,m}$ approximation of almost isometric maps\\[5mm]}
{\today}\\[5mm]
Sergio Conti$^{1}$, Georg Dolzmann$^{2}$, Stefan M\"uller$^{1,3}$ \\[2mm]
{\em $^1$ Institut f\"ur Angewandte Mathematik,
 Universit\"at Bonn,\\ 53115 Bonn, Germany }\\
 {\em $^{2}$ Fakult\"at f\"ur Mathematik, Universit\"at Regensburg,\\
  93040
   Regensburg, Germany}\\
{\em $^3$ Hausdorff Center for Mathematics,
 Universit\"at Bonn,\\ 53115 Bonn, Germany }\\
\bigskip

\begin{minipage}[c]{0.85\textwidth}\small
This paper shows that an almost-isometric Sobolev map from a Lipschitz subset of an oriented manifold $M$ into another oriented manifold $N$
can be approximated by a map of class $C^{1,\alpha}$, with a universal bound on the $W^{2,m}$ norm, for any $m\in(1,\infty)$.
This, in turn,  implies a universal bound on the $C^{1,\alpha}$ seminorm.
The $W^{1,p}$ distance of the approximation from the original map is controlled in terms of the $L^p$-deviation from being an isometry, with the optimal scaling. The manifolds $M$ and $N$ are assumed to be compact, oriented and have equal dimension; $M$ may have a boundary.
\end{minipage}
\end{center}

Keywords: Sobolev functions on manifolds, almost isometric maps, Hölder approximation

MSC codes: 58J05, 53Z30, 74B20

\tableofcontents

\section{Introduction}

Over the last decades,
the theory of nonlinear elasticity has been profoundly influenced by the nonlinear geometric rigidity estimate that states
that for $f \in W^{1,2}(U; \R^n)$, with $U\subset\R^n$
open, bounded, connected,  Lipschitz, one has
\begin{equation} \label{eq:euclidean_rigidity}
\min_{R\in SO(n)} \int_U |Df - R|^2 \, dx \le  C \int_U  \dist^2(Df, SO(n)) \, dx,
\end{equation}
with $C$ depending only on $U$ and $n$ \cite{FrieseckeJamesMueller2002-CPAM}.
Applications range from the derivation of linearized models from nonlinear elasticity \cite{DalmasoNegriPercivale2002,FrierichSchmidt2015,FriedrichKreutzZemas2025} to the derivation of plate theories   \cite{FrieseckeJamesMueller2002-CPAM} and, more in general, the study of low-dimensional reductions of nonlinear elasticity \cite{morafriesecke2003derivation,mora2004nonlinear,FrieseckeJamesMueller2006,LewickaPakzad2011,Mueller2017Review,Lewicka2023Buch}.

It is a natural question whether rigidity depends on the Euclidean structure of $\R^n$. 
Beyond its intrinsic interest from the viewpoint of differential geometry, it is important
for many applications in the derivation of shell theories, see for example the recent work \cite{KupfermanMaorSphere} and references therein.
The left-hand side of \eqref{eq:euclidean_rigidity} is then naturally phrased in terms of the $W^{1,p}$ distance between $f$ and an isometric immersion; the right-hand side is obtained by integrating the pointwise distance between $df$ and the set of isometries between the corresponding tangent spaces.

In the case of a compact oriented manifold $M$ without boundary, and maps from $M$ into itself, the question was answered in \cite{ContiDolzmannMueller2024}, where a bound on the $W^{1,p}$ distance from the set of orientation-preserving isometries on $M$ analogous to \eqref{eq:euclidean_rigidity}
was obtained. One key ingredient was  to replace the map $f$ by a regular approximation $\hat f$, which was obtained using the extrinsic heat flow on $M$.
For related results we refer to \cite{KupfermanMaorSphere}, \cite[Theorems~1.2 and~1.3]{LuckhausZemas2022rigidity}, \cite[Theorem~3.2]{Chen-Li-Slemrod22}.

In the derivation of reduced theories, as in the literature mentioned above, one is naturally led to consider maps defined on general Lipschitz domains. In the Euclidean case, this is already included in \eqref{eq:euclidean_rigidity}. For general bounded, connected Lipschitz subsets of a manifold, however,  the natural generalization of \eqref{eq:euclidean_rigidity} does not hold.

For example, 
let $M=[0,1)^2$ be the two-dimensional torus, 
$U=B_r\subset M$ be a disk of radius $r<1/2$, and let $N$ be 
a rotation surface which contains a flat disk $D_0$ of radius $r_0$ and has positive curvature elsewhere. Any isometry from $U$ into $N$ must map $U$ inside $D_0$.
For $r=r_0$ there is a one-dimensional family of such isometries, which differ by the action of $SO(2)$. If we try to translate one of these isometries by $\delta$, we obtain a map which is isometric away from a set of measure proportional to $\delta$. Since $N$ is smooth, the pointwise distance to isometries grows faster than any power of $\delta$. However, the $W^{1,2}$ distance to the closest isometry is of order $\delta$. Eliminating the cylindrical symmetry, one can obtain similar constructions with a single isometry from $U$ into $N$.
The difficulty lies in the fact that the size of the set of isometries and, correspondingly, the dimension of the space of Killing vector fields, may be different for different translations of $U$.
Therefore, we believe that isometries are not the natural object to study in this setting.

In this paper, we propose uniform approximability with more regular maps as a stable and robust alternative to rigidity in a geometric context.
 For example, we show in Section~\ref{sec:eucllinearization} below that uniform approximability suffices to obtain a linearized model as $\Gamma$-limit of nonlinear elasticity.
In the Euclidean case, the full rigidity \eqref{eq:euclidean_rigidity}  follows immediately from uniform approximability and the corresponding linear estimate, i.e., Korn's inequality (see Section~\ref{sectEucliRigidity} below).
In the Riemannian setting, it is natural to expect that one  can similarly obtain a quantitative rigidity estimate from uniform approximability, at least if the domain and the manifolds obey additional assumptions, this will be explored elsewhere.

Our main result is the following.

\begin{theorem}[Uniform $W^{2,m}$ approximation]  \label{th:W2m_approximation} Suppose that $(M,g)$ and $(N, g_N)$ are smooth, compact, oriented,  $n$-dimensional Riemannian manifolds, $N$ without boundary, and let $\imath: N \to \R^d$ be an isometric embedding.
Let $U \subset M \setminus \partial M$
be an open, bounded, and connected  Lipschitz set, let $p,m  \in (1, \infty)$.
Then there exist  constants $C_1$ and  $C_2$ with the following property. For every map $f \in W^{1,p}(U;N)$
there exists a map $\hat f\in (C^1\cap W^{1,p}\cap W^{2,m})(U; N)$ such that
\begin{equation}\label{eq:mainth:W1p_estimate} 
\| \imath \circ \hat f - \imath \circ f\|_{W^{1,p}(U)} \le C_1 \| \dist(df, SO(M,N)) \|_{L^p(U)}  
\end{equation}
and 
\begin{equation}\label{eq:mainth:W2m_extrinsic}
\| d( \imath \circ \hat f)\|_{L^\infty(U)} + \| \nabla d( \imath \circ \hat f)\|_{L^m(U)}  \le  C_2.  
\end{equation}
\end{theorem}
The notation is explained in Section~\ref{sec:prelim} below; in particular 
$\nabla d( \imath \circ \hat f)$ denotes the covariant derivative of the $\R^d$-valued one-form 
$d( \imath \circ \hat f)$. An intrinsic formulation, which does not make use of the embedding $\imath$, can be obtained 
by using the Sasaki distance  on the tangent bundle
$TN$ \cite{ConventVanSchaftingen2016intrinsic},
but we prefer to use  the extrinsic norm for simplicity of notation.
The argument depends crucially on compactness of $\overline U$ and $N$, but not on compactness of $M$.

In the Euclidean case a similar approximation result with affine maps follows immediately from~\eqref{eq:euclidean_rigidity}. In the Riemannian setting there is no natural analogue of affine maps, solutions of $\nabla d f=0$ are too special if one works with general Lipschitz sets, as the class is unstable with respect to variations of the domain, as discussed above. Therefore we aim for approximation with uniformly $W^{2,m}$ maps.
The approximation in Theorem~\ref{th:W2m_approximation} is obtained via a local construction and in particular holds without any assumption on the family of isometric embeddings of $U$ into $N$.

We briefly sketch the structure of the proof.
We first reduce to the case of uniformly Lipschitz maps (\Cref{pr:lipschitz_approximation}).
The key argument is local, and contained in \Cref{prop:W2m_local}; \Cref{th:W2m_approximation} follows from it by an appropriate covering. The proof in \Cref{prop:W2m_local} is obtained working in local coordinates in both manifolds. We use the Piola identity 
(see
\cite{KupfermanMaorShachar2019}, \cite[Sect.~4]{ContiDolzmannMueller2024}, 
\Cref{prop:weakpiola} below) to show that $f$ is an approximate solution of the harmonic map equation. A first regularization $f_*$ is constructed using a fixed-point argument
in \Cref{prop:harmonic_replacement_interior}, it turns out to be an exact solution of the harmonic map equation in $U$. Using uniform elliptic regularity estimates, one obtains bounds of the form
\begin{equation}\label{eqd2fintro}
 |D^2f_*|\le \frac{A}{\dist(\cdot,\partial U)} + B
\end{equation}
with $A$, $B$ bounded in $L^m$, and $A$ small in $L^p$. Using singular integrals,
in Proposition~\ref{propLipschitzboundaryfull} one decomposes $f_*$ into two parts, one of which has gradient which is small in $L^p$, the other one is bounded in $W^{2,m}$. The second one has the desired properties.

We discuss general background material and notation in Section~\ref{sec:prelim}.
The local argument and how it implies the main proof are discussed in Section~\ref{secmainproof}. In Section~\ref{sec:elliptic}
we review elliptic estimates with attention to uniformity of the estimates, and in 
Section~\ref{sec:solharmmap} we present  the fixed-point argument for solving the harmonic map equation, and the natural bounds which then lead to~\eqref{eqd2fintro}.
The singular-integral estimates are presented in Section~\ref{se:smith}.
Finally, in Section~\ref{sec:implEucl} we discuss for illustration the above-mentioned consequences of our result in the Euclidean case.

\section{Preliminaries}\label{sec:prelim}

We use standard notation for manifolds; see, for example,  \cite{Cheeger-Ebin,Kobayashi-Nomizu63,Kobayashi-Nomizu69,ONeill83,LeeIntroSmoothManifolds2013}.
 The distance in \eqref{eq:mainth:W1p_estimate} 
is defined as 
\begin{equation}\label{eq:defSOMN}
  \dist(df, SO(M,N))(q) := \dist(df(q), SO(T_q M, T_{f(q)} N)),
\end{equation}
where the tangent spaces $T_q M$ and $T_{f(q)}N$ are viewed as oriented
Euclidean spaces equipped with the Riemannian inner product, $SO(V,W)$ denotes the
set of linear and orientation-preserving {isometries} between oriented Euclidean spaces $V$ and $W$,
and the distance is taken with respect to the Hilbert-Schmidt norm on $W \otimes V^* \simeq \Lin(V,W)$.

The Sobolev space $W^{k,s}(U)$, $k\in\N$, can be defined using local charts,
and $W^{k,s}(U;\R^d)$ working componentwise. In turn,
$W^{k,s}(U;N)$ consists of the maps $f:U\to N$ such that $\imath\circ f\in W^{k,s}(U;\R^d)$.
For a vector field $Y:U\to M$  we assume implicitly that $Y(q)\in T_qM$, and similarly for
 related objects.

As in \cite{ContiDolzmannMueller2024},
for $q\in N$ we define
the second fundamental form $A(q) : T_{\imath(q)}\imath(N) \times T_{\imath(q)}\imath(N) \to T^\perp_{\imath(q)}\imath(N)$
 by setting, for $Y\in C^\infty(\imath(N);T\imath(N))$
\begin{equation}
\label{eq:definition_second_fundamental}
 A(q)(v, w) := - ((d_v Y)(\imath(q)))^\perp,  \quad \text{if $Y(\imath(q)) = w$.}
\end{equation}
 We extend $A$ to a  symmetric bilinear form on linear maps from $T_p M$ to $T_{\imath(q)}\imath(N)$ as follows.
 For $p\in M$, $q\in N$, we define
  \begin{equation}
\mathbb A(p,\imath(q)):  (T_{\imath(q)} \imath(N)  \otimes T^*_p M) \times  (T_{\imath(q)} \imath(N)  \otimes T^*_p M)
   \to T^\perp_{\imath(q)} \imath(N)\subset\R^d   
  \end{equation}
by
   \begin{equation} \label{eq:definition_second_endos}
 \mathbb A(p,\imath(q))(X, Y) := \Tr A(X,Y) := g^{\alpha \beta}(p) A(q)(X e_\alpha, Y e_\beta)
\end{equation}
where $e_1, \ldots, e_n$ is a basis of $T_p M$ and $g^{\alpha \beta}$ is the inverse of
$g_{\alpha \beta} := g(e_\alpha, e_\beta)$. We remark that in \cite{ContiDolzmannMueller2024} the dependence of $\mathbb A$ on $p$ was not made explicit.
We denote by $e$ the Euclidean metric of $\R^d$ and the corresponding inner product. Throughout the paper we use the summation convention.
We denote by $B^{(M)}_r(x):=\{y\in M: \dist_M(x,y)<r\}$, or simply by $B_r(x)$ when the ambient manifold is clear from the context, and open neighbourhoods
of sets by $(A)_\delta:=\{x: \dist(x,A)<\delta\}$.
We denote by
$\R^{n\times n}_{+}$ the set of positive definite symmetric matrices.
We use $D$ for the differential of maps from $\R^n$ to $\R^m$.

For $p\in M$, $q\in N$,  $X\in T_{\imath(q)}\imath(N)\otimes T^*_pM$ we define $\Det X \in \R$ by
 \begin{equation}
 X^* \vol_{\imath(N)} = \Det X \,  \, \vol_M.
 \end{equation}
Then $t \mapsto \Det (X + tY)$ is a polynomial and we define the cofactor operator $\Cof X \in T_{\imath(q)}\imath(N)\otimes T^*_pM$ by
\begin{equation}  \label{eq:define_Cof}
\langle \Cof X, Y \rangle_{g,e} := \frac{d}{dt}|_{t=0} \Det(X + tY).
\end{equation}

The Piola identity $\Div\cof Du=0$ is a key ingredient for the Euclidean proof \cite{FrieseckeJamesMueller2002-CPAM}, coupled with the fact that on bounded sets 
$|A-\cof A|\le C\dist(A, SO(n))$ for $A$ bounded.
After introducing the appropriate geometric objects,
these two facts continue to hold on  manifolds. One obtains the following (see
\cite{KupfermanMaorShachar2019}, \cite[Sect.~4]{ContiDolzmannMueller2024}).

\begin{proposition}\label{prop:weakpiola}
Let $M$, $N$, $U$, $\imath$ be as in Theorem~\ref{th:W2m_approximation}.
Let $f\in \Lip(U;N)$, set $\hat f:=\imath\circ f\in \Lip(U;\R^d)$. Then there are $\hat h\in L^\infty(U; T\imath(N)\otimes T^*M)$ and $\hat h'\in L^\infty(U; T^\perp\imath(N))$ such that
 for any $\xi\in W^{1,\infty}_0(U; \R^d)$
one has  \begin{equation}\label{eqpiolacofprelim}
  \int_U \langle d\hat f, d\xi\rangle_{g,e}
  -\langle(\mathbb A\circ (\id, \hat f))(d\hat f,d\hat f), \xi\rangle_e d\vol_M
  =\int_U -\langle \hat h, d\xi\rangle_{g,e}
  +\langle \hat h',\xi\rangle_e d\vol_M.
 \end{equation}
 More precisely, they
can be chosen as
\begin{equation}\label{eqboundstheo21}
 \begin{split}
  \hat h:=& \Cof d\hat f-d\hat f, \\
  \hat h':=& -(\mathbb A\circ (\id, \hat f)) (d\hat f-\Cof d\hat f,d\hat f).
 \end{split}
\end{equation}
\end{proposition}
\begin{proof}
 This follows from the extrinsic Piola identity proven in \cite{KupfermanMaorShachar2019}
 as discussed in \cite[Sect.~4]{ContiDolzmannMueller2024}.
\end{proof}
The left-hand side of~\eqref{eqpiolacofprelim} is the same expression as in the harmonic map equation (indeed, the harmonic map equation arises from taking variations of $\int |df|^2d\vol_M$, the Piola identity from taking variations of $\int \Det df d\vol_M$). One can therefore write it intrinsically and obtain
\begin{equation}\label{eqdeltagprelim1}
-\Delta_g f =\Gamma^N(df,df)+\Div_g  h +h'
\end{equation}
in the sense of distributions,
and using that $|d\hat f-\Cof d\hat f|\le C \dist(df,SO(M,N))$ from \eqref{eqboundstheo21} one has
\begin{equation}\label{eqhhprimeprelim2}
 |h| + |h'|\le  C   \dist(df,SO(M,N))\quad \text{ pointwise in $U$},
\end{equation}
with $C$ depending on $\Lip(f)$. Since $f$ is Lipschitz,
the right-hand side of~\eqref{eqhhprimeprelim2} is itself bounded.
As in the usual notation for the harmonic map equation, 
here $\Delta_g$ is the Laplace-Beltrami operator, 
and $\Gamma^N(df(p), df(p))$ is a vector-valued quadratic form on $\Lin(T_pM, T_{f(p)} N)$ which is constructed from the metric and its derivatives, and corresponds to $\mathbb A$ in~\eqref{eqpiolacofprelim}.
We denote by $\Lip(f)$ the Lipschitz constant of $f$, and by $[f]_\alpha$ the $\alpha$-Hölder seminorm.

For the convenience of the reader, we give a short self-contained derivation of
the equation we shall actually use, which is
\eqref{eqdeltagprelim1} in coordinates, without using the extrinsic version in Proposition~\ref{prop:weakpiola}.

\begin{proposition}[Piola in coordinates]\label{proppiolacoordinates}
Let $\Lambda>0$, $C_0>0$. Then there is $C=C(n,p,\Lambda,C_0)$ with the following property.

 Let $M$, $N$ oriented $n$-dimensional Riemannian manifolds,
 $O\subset M\setminus \partial M$, $O'\subset N$
 open,
 $\varphi:O\to\R^n$, $\psi:O'\to\R^n$ charts.
 Assume that the coordinate representations of the metrics $g$ and $g_N$ obey
 \begin{equation}\label{eqgC0p}
 \frac1{C_0}\Id \le g\le C_0 \Id, \hskip5mm
 |Dg|+|D^2g|+|D^3g|\le C_0,
\end{equation}
\begin{equation}\label{eqgNC0}
 \frac1{C_0}\Id \le g_N\le C_0 \Id, \hskip5mm
 |Dg_N|+|D^2g_N|+|D^3g_N|\le C_0
\end{equation}
pointwise in $\varphi(O)$ and $\psi(O')$.
Then there is a bilinear map
\begin{equation}
 Q: \varphi(O)\times \psi(O')\to {\textrm Bil}(\R^{n\times n}\times \R^{n\times n};\R^n)
\end{equation}
with
\begin{equation}\label{eqboundsQbilin}
 |Q|+|DQ|+|D^2Q|\le C
\end{equation}
such that  for any open set $U\subset O$ and any $f:U\to O'$ with $\Lip(f)\le\Lambda$ there are $H:\varphi(U)\to \R^{n\times n}$,
$H':\varphi(U)\to \R^{n}$ such that
the coordinate representation $F:=\psi\circ f\circ\varphi^{-1}:\varphi(U)\to\psi(O')$ obeys
\begin{equation}\label{eqLgFQhhpprop22}
 -L_g F = Q(\cdot,F)(DF,DF) + \Div H + H'
\end{equation}
distributionally in $\varphi(U)$, with
\begin{equation}\label{eqhhprimepointwise}
 |H| + |H'| \le C \dist(df, SO(M,N))\circ \varphi^{-1} \text{ a.e. in $\varphi(U)$}.
\end{equation}
\end{proposition}
The Laplace-Beltrami operator $L_g$ acts on $F$ componentwise,
\begin{equation}\label{eqdefLgfromgbb}
 (L_g F)_k:=
\frac1{\sqrt {\det g}}\partial_\alpha (\sqrt{\det g} g^{\alpha\beta} \partial_\beta F_k).
\end{equation}
Here $g^{\alpha\beta}$ denotes the entries of the inverse of the matrix $g$,
$g^{\alpha\beta}=(g^{-1})_{\alpha\beta}=g^{-1}_{\alpha\beta}$.
Equation~\ref{eqLgFQhhpprop22} with $h$ and $h'$ equal zero is the harmonic map equation.

Condition~\ref{eqhhprimepointwise} immediately implies
 \begin{equation}
  \|h\|_{L^p(\varphi(U))} + \|h'\|_{L^p(\varphi(U))} \le C \|\dist(df, SO(M,N))\|_{L^p(U)}.
 \end{equation}
We recall that the coordinate representation of the metrics is given by the maps
$g: \varphi(U)\to \R^{n\times n}_\sym$,
$g_N: \psi(O')\to \R^{n\times n}_\sym$, defined
by
\begin{equation}
 g(x) v\cdot w = g_{\varphi^{-1}(x)}(d\varphi^{-1}(x)v, d\varphi^{-1}(x)w)
\end{equation}
and
\begin{equation}
 g_N(y) v\cdot w = g_{N,\psi^{-1}(y)}(d\psi^{-1}(y)v, d\psi^{-1}(y)w),
\end{equation}
where the functions in the right-hand side are the metrics seen as intrinsic objects.
 \begin{proof}
For $q\in U$, we shorten $x:=\varphi(q)$, $y:=\varphi(f(q))=F(x)$.

Let $A\in \Lin(T_qM, T_{f(q)}N)$.
The definition in \eqref{eq:defSOMN} means that there is $Q$  in the same space such that
$|Qv|_{T_{f(q)}N}=|v|_{T_qM}$ for all $v\in T_qM$,
and $\dist(A, SO(M,N))=|A-Q|_{\Lin(T_qM;T_{f(q)}N)}$.
The coordinate representation $\hat Q$ of $Q$ then obeys
$g_N\hat Q\hat v \cdot \hat Q\hat v=g \hat v\cdot  \hat v$ for all $\hat v\in\R^n$, which means that  $Q^*:=g_N^{1/2} \hat Qg^{-1/2}\in SO(n)$.
Therefore
for every $x\in \varphi(U)$ there is $Q^*\in SO(n)$ such that
\begin{equation}\label{eqdistDFdfSOnpf}
 |(g_N(F(x)))^{1/2} D F(x) g^{-1/2}(x)-Q^*|\le C
  \dist(df, SO(M,N))(\varphi^{-1}(x)),
\end{equation}
and $Q^*=\cof Q^*$.
We define $G:\varphi(U)\to\R^{n\times n}$ by
\begin{equation}\begin{split}
 G:=&\cof[(g_N\circ F )^{1/2}  D F g^{-1/2}]-(g_N\circ F )^{1/2} D F g^{-1/2},
 \end{split}
\end{equation}
by \eqref{eqdistDFdfSOnpf} we obtain
\begin{equation}\label{eqhpointwisedist}
|G|\le C\dist(df, SO(M,N))\circ\varphi^{-1}
\end{equation}
pointwise, with $C$ depending on $n$, $\Lambda$, $C_0$.
We recall that
in $\R^{n\times n}$ we have $\cof(AB)=\cof A \cof B$, and
(for invertible matrices)
$\cof A = A^{-T}\det A $. Therefore,
writing $g_N$ for $g_N\circ F$ for brevity,
\begin{equation}\begin{split}
 G =&
 \frac{\sqrt{\det g_N}}{\sqrt{\det g}}
 g_N^{-1/2}
 (\cof  D F) g^{1/2}-g_N^{1/2} D F g^{-1/2}\\
 =&
 \frac{\sqrt{\det g_N}}{\sqrt{\det g}}
 g_N^{-1/2}
 [\cof  D F -
 \frac{\sqrt{\det g}}{\sqrt{\det g_N}}
 g_N  D F g^{-1}]g^{1/2}\\
 =& \frac{\sqrt{\det g_N}}{\sqrt{\det g}}
 g_N^{-1/2} \hat G g^{1/2},
 \end{split}
\end{equation}
where
\begin{equation}
 \hat G:=\cof  D F -
 \frac{\sqrt{\det g}}{\sqrt{\det g_N}}
 g_N  D F g^{-1}
\end{equation}
still obeys \eqref{eqhpointwisedist} (with a different $C$).
We remark that all computations done up to now only contain
 $g$, $g_N$, and
first derivatives of $F$ and are therefore valid pointwise almost everywhere, as identities between $L^\infty$ functions.

We know  that $\Div\cof DF=0$ distributionally. Therefore
\begin{equation}\label{eqdivhatadivll}
\Div\hat G = - \Div \left[\frac{\sqrt{\det g}}{\sqrt{\det g_N}}
 g_N  D F g^{-1}\right]
\end{equation}
distributionally,
which is the equation we are after. It remains to rewrite it in a form that contains the Laplace-Beltrami operator $L_g$.  We shall use the product rule to separate the terms with derivatives of $g_N$ (which, we recall, was short for $g_N\circ F$). These terms do not have second derivatives of $F$ but, by the chain rule, have a second factor of $DF$; the rest depends only on $g$ and $F$, up to a positive-definite prefactor that we can move to the other side.

To simplify notation we set $p^N:=g_N/\sqrt{\det g_N}$.
We use greek indices in $\varphi(U)$, and latin ones in $\psi(O')$, both run from 1 to $n$.
Then~\eqref{eqdivhatadivll} reads
\begin{equation}
(\Div\hat G)_j = - \partial_\alpha\Bigl[p^N_{jk}\circ F\,  {\sqrt{\det g }}
  (\partial_\beta F_k)g^{-1}_{\beta\alpha}\Bigr]
\end{equation}
and the product rule gives
\begin{equation}\label{eqdivcof3}
(\Div\hat G)_j= - (\partial_h p^N_{jk})\circ F\,
(\partial_\alpha F_h)
{\sqrt{\det g }}
  (\partial_\beta F_k) g^{-1}_{\beta\alpha}
  - p^N_{jk}\circ F\, \partial_\alpha\Bigl[  {\sqrt{\det g }}
  g^{-1}_{\alpha\beta} \partial_\beta F_k \Bigr].
\end{equation}

We recall~\eqref{eqdefLgfromgbb}.
We divide~\eqref{eqdivcof3}  by $\sqrt{\det g}$, multiply by the inverse matrix of $p^N\circ F$, and rearrange terms to obtain
\begin{equation}\label{eqdivcof4}
  - L_g F_k=
[  (p^N)^{-1}_{kj} \partial_h p^N_{jk'}]\circ F\,
   g^{-1}_{\alpha\beta}
    (\partial_\alpha F_h)(\partial_\beta F_{k'})
  +\frac1{\sqrt{\det g }}
(p^N)^{-1}_{kj}\circ F
  (\Div\hat G)_j.
\end{equation}
The first term on the right-hand side can be rewritten as
$Q_k(x,F(x))(DF,DF)$, with $Q(x,y)$ an $\R^n$-valued bilinear form
on $\R^{n\times n}$
that depends only on the two metrics. Specifically,
\begin{equation}
 Q_k(x,y)(A,B)
:= [  (p^N)^{-1}_{kj} \partial_h p^N_{jl}](y)\,
   g^{-1}_{\alpha\beta}(x)
    A_{h\alpha} B_{l\beta}.
\end{equation}
Inserting the definition of $p^N$,
\begin{equation}
 Q_k(x,y)(A,B)
= [  \sqrt{\det g_N} (g_N)^{-1}_{kj} \partial_h
\frac{(g_N)_{jl}}{\sqrt {\det g_N}}](y)\,
   g^{-1}_{\alpha\beta}(x)
    A_{h\alpha}B_{l\beta};
\end{equation}
since this expression only contains first derivatives of the metrics, the bounds in \eqref{eqboundsQbilin} follow.
In turn, the last term in \eqref{eqdivcof4} can be rewritten as $H'+\Div H$, with
$\|H\|_p+\|H'\|_p\le C\eps$. We obtain
\begin{equation}\label{eqdivcof4b}
\frac1{\sqrt{\det g }}
(p^N)^{-1}_{kj}\circ F
  (\Div\hat G)_j= \partial_\alpha H_{k\alpha} + H'_k
\end{equation}
where
\begin{equation}
 H_{k\alpha}:=\frac{\sqrt{\det g_N\circ F}}{\sqrt{\det g }}
(g_N)^{-1}_{kj}\circ F
  \, \hat G_{j\alpha}
\end{equation}
and
\begin{equation}
 H'_k:=-\hat G_{j\alpha} (p^N)^{-1}_{kj}\circ F
 \partial_\alpha \frac1{\sqrt{\det g }}
 -\hat G_{j\alpha} \frac1{\sqrt{\det g }}
  (\partial_l \frac{(g_N)^{-1}_{kj}}{\sqrt{\det g_N}})\circ F
  \partial_\alpha F_l
\end{equation}
obey $|H|+|H'|\le C |\hat G|$ pointwise, and therefore
\eqref{eqhhprimepointwise}.
\end{proof}

A second ingredient is the following by now standard truncation result. The proof of this version follows, for example, from Step IV in the proof of Theorem 3  in \cite[pp. 390--392]{KupfermanMaorShachar2019} or \cite[Step 1 in the proof of Theorem~4.1(i)]{KroemerMueller2021}, see also
\cite[Prop.~3.1]{ContiDolzmannMueller2024}.
\begin{proposition}\label{proptruncation} \label{pr:lipschitz_approximation}
Let $M$, $N$, $U$, $\imath$ be as in Theorem~\ref{th:W2m_approximation}, $p\in (1,\infty)$. There exist constants $\Lambda$, $C > 0$ with the following property.
If $f \in W^{1,p}(U; N)$ then there exists $\tilde f \in \Lip(U;N)$ such that
\begin{equation}
\Lip(\tilde f) \le \Lambda
\end{equation}
and
\begin{equation}\label{eqlipspprox2}
 \|\imath \circ f -\imath\circ\tilde f\|_{W^{1,p}(U)} \le C \|\dist(df, SO(M,N)) \|_{L^p(U)}.
\end{equation}
In particular,
 \begin{equation}\label{eqlipspprox3}
 \|\dist(d\tilde f, SO(M,N)) \|_{L^p(U)}
\le C \|\dist(df, SO(M,N)) \|_{L^p(U)}.
 \end{equation}
\end{proposition}

We next define what we mean by a Lipschitz subset of a manifold.
Recall that a set $A\subset M$ is open if for any $x\in A$ there is $r>0$ such that $B_r^{(M)}(x)\subset A$.
This means that, possibly after reducing $r$,
there is a chart $\varphi:B_r^{(M)}(x)\to \R^n$ such that the following holds:
if $x\in M\setminus \partial M$,
$\varphi(A\cap B_r^{(M)}(x))$ is open in $\R^n$; if $x\in \partial M$, then
$\varphi(A\cap B_r^{(M)}(x))$ is relatively open in $\R^n_-:=\R^n\cap\{x_n\le0\}$.
As usual in this context, for $x\in\R^n$ we write $x=(x',x_n)$, $x'\in\R^{n-1}$, $x_n\in\R$.

\begin{definition}\label{defLipschitz}
A subset $U$ of a manifold $M$ is an open, bounded Lipschitz set if $U\subset M\setminus\partial M$ is open,
$\overline U\subset M$ is compact,  and
for every $x\in \partial U$ there are
an open set $O\subset M$  with $x\in O$, a chart $\varphi\in C^\infty(O;\R^n)$, an open set $\omega\subset\R^n$ with $\varphi(x)\in\omega$, and a function $\gamma\in\Lip(\R^{n-1})$ such that
\begin{equation}\label{eqdeflipschitzboundary}
 \omega \cap \varphi(O\cap U) = \omega\cap \{y\in \R^n: y_n<\gamma(y')\}.
\end{equation}
\end{definition}
We first recall
that
if this  property holds for one chart, then it holds for all charts, up to an affine change of variables (see also \cite[Sect.~4.1]{HofmannMitreaTaylor2007}).
\begin{remark}
Let $U$ be as in Definition~\ref{defLipschitz}, let $\hat\varphi\in C^\infty(\hat O;\R^n)$ be a chart, $x\in\partial U\cap  \hat O$. Then
there are an affine bijective map $I:\R^n\to\R^n$, $\hat\omega\subset\R^n$ open with $\hat\varphi(x)\in\hat\omega$, $\hat\gamma\in\Lip(\R^{n-1})$, and an open set $O'\subset \hat O\subset M$ with $x\in O'$ such that
\begin{equation}\label{eqdeflipschitzboundaryhat}
 \hat{\omega}  \cap \hat{\varphi}(O'\cap U)  = \hat{\omega}\cap I \{y\in \R^n: y_n<\hat{\gamma}(y')\}.
\end{equation} 
\end{remark}
An elementary proof is based on  Lemma~\ref{lemmalipschitzset} below.
Indeed,
fix $O$, $\varphi$, $\omega$, $\gamma$ as in~\eqref{eqdeflipschitzboundary}.
Since the charts $\varphi$ and $\hat\varphi$ are compatible, there
are an open set
$O'\subset O\cap \hat O$ with $x\in O'$,
an open set $\omega'\subset\R^n$ with $\hat\varphi(O')\subset\omega'$,
and a smooth diffeomorphism $\Psi:\omega'\to\Psi(\omega')\subset\R^n$ 
such that $\varphi=\Psi\circ\hat\varphi$ on $O'$.
We pick $\omega''\subset \omega'$ open with $\hat\varphi(x)\in\omega''$ and 
$\Psi(\omega'')\subset\omega$.  Equation~\eqref{eqdeflipschitzboundary} implies
\begin{equation}\label{eqdeflipschitzboundaryOp}
\Psi(\omega'')\cap \varphi(O'\cap U) = \Psi(\omega'')\cap \{ y_n<\gamma(y')\}.
\end{equation}
By Lemma~\ref{lemmalipschitzset}, possibly reducing $\omega''$ 
(still with $\hat\varphi(x)\in\omega''$)
we see that there are an affine isomorphism $I:\R^n\to\R^n$ and $\Gamma\in \Lip(\R^{n-1})$ with
 \begin{equation}\label{eqlipchvar}
  \Psi(\omega'' \cap I 
  \{y_n<\Gamma(y')\})
  = \Psi(\omega'')\cap \{y_n<\gamma(y')\}.
 \end{equation}
Combining~\eqref{eqlipchvar} 
and~\eqref{eqdeflipschitzboundaryOp} gives
 \begin{equation}
     \Psi(\omega'') \cap \varphi(O'\cap U) 
=  \Psi(\omega'' \cap I 
  \{y_n<\Gamma(y')\}).
 \end{equation}
We insert $\varphi=\Psi\circ\hat\varphi$ and use bijectivity of $\Psi$ to conclude
 \begin{equation}
   \omega'' \cap \hat\varphi(O'\cap U) 
   = \omega'' \cap I 
  \{y_n<\Gamma(y')\}
   ,
 \end{equation}
which concludes the proof of \eqref{eqdeflipschitzboundaryhat}.

\begin{lemma}\label{lemmalipschitzset}
 Let $\omega$, $\Omega\subset\R^n$ be open,
 $\Psi\in C^1(\Omega;\omega)$ be a diffeomorphism,
 $L\ge 1$, $z\in\omega$.

 Then there are an open set $R$ with
 $\Psi^{-1}(z)\in R \subset\Omega$
 and an affine isomorphism $I:\R^n\to\R^n$
 such that for any $\gamma\in\Lip(\R^{n-1})$ with $\Lip(\gamma)\le L$ and $\gamma(z')=z_n$ there is $\Gamma\in \Lip(\R^{n-1})$ with $\Lip(\Gamma)\le 4L$ such that
 \begin{equation}\label{eqliiptransfsetpsi}
  \Psi(R \cap I 
  \{x\in \R^n: x_n<\Gamma(x')\})
  = \Psi(R) \cap \{y\in \R^n: y_n<\gamma(y')\}.
 \end{equation}
\end{lemma}
\begin{figure}
 \begin{center}
\includegraphics[width=12cm]{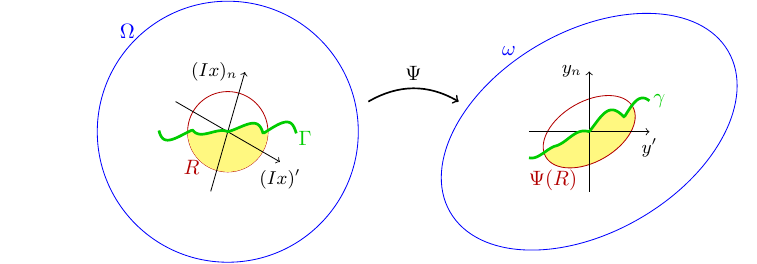}
 \end{center}
\caption{Sketch of the geometry in Lemma~\ref{lemmalipschitzset}. The points in the set $\Omega$ in the left are labeled $Ix$, with $x$ being the local coordinates representation, in which the curve $\Gamma$ is defined. The points on the right are labeled $y$ and relate to the curve $\gamma$.}
\end{figure}
Setting $z^*:=I^{-1}\Psi^{-1}(z)$ one obtains 
$z^*_n=\Gamma((z^*)')$.
\begin{proof}
We can assume $z=0$.
We choose an affine isomorphism $I:\R^n\to\R^n$ such that
$\hat\Psi:=\Psi\circ I$,
$\hat\Psi\in C^1(\hat\Omega;\omega)$,
with $\hat\Omega:=I^{-1}(\Omega)$,
obeys
  $\hat\Psi(0)=0$
 and $D\hat\Psi(0)=\Id$
 (this means $I(x)=\Psi^{-1}(0)+D(\Psi^{-1})(0)x$). It suffices to prove the assertion for $\hat\Psi$ with $I$ replaced by the identity. For notational simplicity we denote $\hat\Omega$ and $\hat\Psi$ by $\Omega$ and $\Psi$.

Let $\eta_L:=1/(2+2L)$.
We choose $\rho>0$ so  that
$R_\rho:=
 B'_\rho\times (-4L\rho,4L\rho)\subset \Omega$ and
$|D\Psi-\Id|\le \eta_L$ in  $R_\rho$.
This implies that
\begin{equation}\label{eqPsiIdlip}
 |\Psi(x)-\Psi(\hat x)-(x-\hat x)|
 \le \eta_L |x-\hat x| \hskip1cm
\text{
 for any $x,\hat x\in R_\rho$},
 \end{equation}
 so that in particular $\Psi$ is $(1+\eta_L)$-Lipschitz in this set.

Fix $x'\in B'_\rho$. We seek $\Gamma(x')$ such that
$\Psi_n((x',\Gamma(x')))=\gamma(\Psi'((x',\Gamma(x'))))$.
We define $\Gamma(x')$ by showing that the function
\begin{equation}\label{eqdefvarphilip}
 \varphi(t):=\varphi_{x'}(t):=\Psi_n((x',t))- \gamma(\Psi'((x',t)))
\end{equation}
has a unique zero in $(-4L\rho,4L\rho)$.
To do this, for $t,s$ in this interval we evaluate
\begin{equation}
 \varphi(t)- \varphi(s)=\Psi_n((x',t))- \Psi_n((x',s))-\gamma(\Psi'((x',t)))+\gamma(\Psi'((x',s)))
\end{equation}
and observe that by \eqref{eqPsiIdlip} we have
\begin{equation}\begin{split}
| \varphi(t)- \varphi(s)-(t-s)| \le& \eta_L |t-s| + L |\Psi'((x',t))-\Psi'((x',s))|\\
\le& \eta_L |t-s| + L \eta_L |t-s|\le \frac12 |t-s|,
\end{split}\end{equation}
so that for $-4L\rho\le s<t\le 4L\rho$ we have
\begin{equation}\label{eqvarphimon}
 \frac12 (t-s)\le
  \varphi(t)- \varphi(s)
  \le \frac32 (t-s).
\end{equation}
Therefore $\varphi$ is continuous and strictly increasing. To prove that it has a unique zero, it suffices to show that it changes sign. For this, we estimate, recalling $\Psi(0)=0$ and $\gamma(0)=0$,
\begin{equation}\begin{split}
 |\varphi(0)|\le &|\Psi_n((x',0))|+ |\gamma(\Psi'((x',0)))|
 \le \eta_L |x'|+ L (1+\eta_L) |x'|\\
 < & (L+1)\rho\le 2L\rho.
\end{split}\end{equation}
Using twice the first inequality in \eqref{eqvarphimon},
\begin{equation}
 \varphi(-4L\rho)
 \le -2L\rho + \varphi(0)<0,
\hskip1cm
 \varphi(4L\rho)
 \ge 2L\rho + \varphi(0)>0.
\end{equation}
Therefore $\varphi$ has a unique zero, which defines $\Gamma(x')$.

If $x_n<\Gamma(x')$,
which means that $\Psi(x)$ is in the set on the left of \eqref{eqliiptransfsetpsi},
then $\varphi(x_n)<0$,
which means that
$\Psi_n(x)< \gamma(\Psi'(x))$,
so that $\Psi(x)$ is also in the set on the right. A similar argument holds for the other inequality, therefore the two sets are equal.

It remains to show that $\Gamma$ is a Lipschitz function.
Consider two points $x'$ and $z'\in B'_\rho$,
let $t:=\Gamma(x')$, $s:=\Gamma(z')$.
These definitions mean that
\begin{equation}
 0=\Psi_n((x',t))-\gamma(\Psi'((x',t)))
 =\Psi_n((z',s))-\gamma(\Psi'((z',s))).
\end{equation}
Again by \eqref{eqPsiIdlip},
\begin{equation}
|\Psi_n((x',t))-
\Psi_n((z',s))- (t-s)|
\le \eta_L(|x'-z'|+|t-s|)
\end{equation}
and
\begin{equation}\begin{split}
|\gamma(\Psi'((x',t)))
-\gamma(\Psi'((z',s)))|
\le& L
|\Psi'((x',t))
-\Psi'((z',s))|\\
     \le& L(|x'-z'|+\eta_L (|x'-z'|+|t-s|)).
\end{split}    \end{equation}
Combining the two,
\begin{equation}
 (1-\eta_L-L\eta_L)|t-s|\le (L+\eta_L+L\eta_L) |x'-z'|
\end{equation}
and recalling the definition of $\eta_L$
\begin{equation}
\frac12|t-s|\le 2L |x'-z'|,
\end{equation}
so that $\Gamma$ is $4L$-Lipschitz.

From the definition in \eqref{eqdefvarphilip}, it follows immediately that $\Gamma(0)=0$.

\end{proof}

\section{Proof of the main theorem}\label{secmainproof}

The proof is based on a local result stated in Proposition~\ref{prop:W2m_local}.
We first show how \Cref{th:W2m_approximation} follows from
Proposition~\ref{prop:W2m_local}, and in the rest of this Section we give a proof of the Proposition. In turn, this uses the results that are presented in the following sections.

\newcommand{\qinM}{q}
\newcommand{\QinM}{Q}
\begin{proposition}\label{prop:W2m_local} 
Let $M$, $N$, $U$, $\imath$ be as in Theorem~\ref{th:W2m_approximation}.
Let $p,m  \in (1, \infty)$, $\Lambda>0$, $\delta>0$.
Fix open sets
$O'\subset\subset O\subset N$,   a chart $\psi:O\to\R^n$,  $\qinM\in \overline U$.

Then there are $C>0$,
open sets $\omega'\subset\subset \omega\subset M$, with $\qinM\in\omega'$,
such that the following holds:

For any $f:\omega\cap \overline U\to O$ with $f(\qinM)\in O'$ and $\Lip(f)\le\Lambda$, there exists
a map $\hat f\in (C^1\cap W^{2,m}\cap W^{1,p})(\omega'\cap U; O)$ such that
\begin{eqnarray}
\| \imath\circ \hat f - \imath\circ f\|_{W^{1,p}(\omega'\cap U)} &\le &C \| \dist(df, SO(M,N)) \|_{L^p(\omega\cap U)},  \label{eq:W1p_estimateimath} \\
\| d(\imath\circ \hat f)\|_{L^\infty(\omega'\cap U)} &+& \| \nabla d (\imath\circ \hat f)\|_{L^m(\omega'\cap U)}
 \le  C,  \label{eq:W2m_extrinsicimath}
\\
 \|\imath\circ \hat f-\imath\circ f\|_{L^\infty(\omega'\cap U)} &\le &  \delta.
 \label{eq:Linfty_extrinsicimath}
\end{eqnarray}
\end{proposition}

\begin{proof}[Proof of Theorem~\ref{th:W2m_approximation} from Proposition~\ref{prop:W2m_local}]
\mbox{}

\begin{enumerate}
\item   \emph{Reduction to the case that $f$ is $\Lambda$-Lipschitz.}

Let $\Lambda$ be the constant from
Proposition~\ref{proptruncation}.
For any $f$ as in the Theorem, there is
a $\Lambda$-Lipschitz map $\tilde f:U\to N$
satisfiying
\eqref{eqlipspprox2} and \eqref{eqlipspprox3}. One easily checks that it suffices to prove the assertion for $\tilde f$, which is denoted by $f$ in the following.

Therefore it suffices to prove the assertion under the additional assumption
\begin{equation}
 \text{  $\Lip(f)\le\Lambda$}.
\end{equation}

 \item \emph{Coverings.}
\label{itemchoicecovering}

For every point $y\in N$ we pick $r_y>0$ such that there is a chart $\psi_y: B_{r_y}^{(N)}(y)\to\R^n$. By compactness we can select finitely many points $Y:=\{y_1, \dots, y_K\}$ such that the smaller balls $\{B^{(N)}_{r_y/2}(y)\}_{y\in Y}$ cover $N$.
Let $r_0:= \min\{r_y: y\in Y\}$.

Let $\delta_N>0$ be such that there is a well-defined and smooth projection
$\pi:(\imath(N))_{\delta_N}\to \imath(N)$.
For any $y\in Y$ and $\qinM\in\overline U$ we apply
Proposition~\ref{prop:W2m_local},
with $\Lambda$  as fixed in Step (i),
 $\delta=\delta_N/2$,
$O'=B_{r_y/2}(y)$,
 $O=B_{r_y}(y)$, $\psi=\psi_y$.
We obtain that
there are open sets
$\omega'_{\qinM,y}$, $\omega_{\qinM,y}$
and constants $C_{\qinM,y}>0$ with the property stated.

The set \begin{equation}\label{eqdefomegaqinm}
\omega'_\qinM:=
B_{r_0/(2\Lambda)}^{(M)}(\qinM)\cap
\bigcap_{y\in Y} \omega'_{\qinM,y}
    \end{equation}
is a neighbourhood of $\qinM$ (in the topology of $M$).
Since these sets form an open cover of the compact set $\overline U$, there is a finite set $\QinM\subset\overline U$ such that
$\overline U\subset \cup_{\qinM\in \QinM} \omega'_\qinM$.
Let $C:=\max_{\qinM\in \QinM, y\in Y} C_{\qinM,y}$.

\item\emph{Construction of $\hat f$.}

Let $f:U\to N$ be $\Lambda$-Lipschitz.
For every $\qinM\in \QinM$, there is $y=y_\qinM\in Y$ such that
$f(\qinM)\in B_{r_y/2}(y)$.
By~\eqref{eqdefomegaqinm}, $f(\omega'_q\cap U)\subset B_{r_0/2}(f(q))
\subset O=B_{r_y}(y)$.
By
Proposition~\ref{prop:W2m_local} there is a function $\hat f_\qinM\in (C^1\cap W^{1,p}\cap W^{2,m})(\omega'_q\cap U;O)$ such that
\begin{eqnarray}
\label{eq:W1p_estimateimathcopy}
\| \imath\circ \hat f_\qinM - \imath\circ f\|_{W^{1,p}(\omega'_q\cap U)} &\le &C \| \dist(df, SO(M,N)) \|_{L^p(\omega_q\cap U)},   \\
\| d (\imath\circ \hat f_\qinM)\|_{L^\infty(\omega'_q\cap U)} &+& \| \nabla d (\imath\circ \hat f_\qinM)\|_{L^m(\omega'_q\cap U)}
 \le  C,  \label{eq:W2m_extrinsicimathcopy}
\\
\label{eq:propLinftyboundcopy}
\|\imath\circ \hat f_\qinM-\imath\circ f\|_{L^\infty(\omega'_q\cap U)} &\le &  \delta.
\end{eqnarray}
Since $\overline U$ is a compact set covered by the open sets $\omega'_\qinM$, there is a partition of unity $\theta_\qinM\in C^\infty_c(\omega'_\qinM;[0,1])$ with $\sum_{\qinM\in \QinM}\theta_\qinM=1$ on $\overline U$.
We remark that $\theta_q$ does not need to vanish on $\partial\omega_q'\cap\partial M$.
We define $\FPhi:U\to \R^d$ by
\begin{equation}
 \FPhi(u):=\sum_{\qinM\in \QinM} \theta_\qinM(u)\imath\circ \hat f_\qinM(u) .
\end{equation}
We next project to $\imath(N)$.
In order to use projection, we need to
show  that $\dist(\FPhi(u),\imath(N))<\delta_N$ uniformly, with $\delta_N$ as in Step \ref{itemchoicecovering}.
By \eqref{eq:propLinftyboundcopy},
\begin{equation}
|\imath\circ \hat f_q(u)-
\imath\circ  f(u)|\le
\delta \text{ for all $u\in \omega'_q\cap U$}
\end{equation}
and by convexity,
\begin{equation}\label{eqdistPhihatLinfty}
|\FPhi(u)-\imath\circ f(u)|\le
\sum_{\qinM:\theta_\qinM(u)>0} \theta_\qinM(u) |\imath\circ \hat f_\qinM(u)-\imath\circ  f(u)|\le \delta.
\end{equation}
In particular,
\begin{equation}
\dist(\FPhi(u),\imath(N))\le
|\FPhi(u)-\imath\circ f(u)|\le \delta.
\end{equation}
At this point we can define $\hat f:=\imath^{-1}\circ\pi\circ \FPhi$.

\item \emph{Bounds on $\hat f$.}

Let $\Phi:=\imath\circ f$. We compute, using \eqref{eq:W1p_estimateimathcopy},
\begin{equation}\label{eqphihatphiLp}
\begin{split}
 \|\Phi-\hat\Phi\|_{L^p(U)}
 &=\|\sum_q \theta_q (\imath\circ f-\imath\circ  \hat f_q)\|_{L^p(U)}\\
&\le  \sum_q  \|\imath\circ f-\imath\circ  \hat f_q\|_{L^p(\omega_q'\cap U)}\le
 C\eps,\end{split}
\end{equation}
where
\begin{equation}
 \eps:=\|\dist(df, SO(M,N))\|_{L^p(U)}.
\end{equation}
Similarly,
\begin{equation}\label{eqdphihatdphiLp}\begin{split}
 \|d \Phi- d\hat\Phi\|_{L^p(U)}
 &=\|
   \sum_q d (\theta_q\Phi)
-\sum_q d (\theta_q \imath\circ \hat f_q)\|_{L^p(U)}\\
&=\|\sum_q d [\theta_q (\imath\circ f-\imath\circ  \hat f_q)]\|_{L^p(U)}\\
&\le \|\sum_q |d \theta_q| \, |\imath\circ f-\imath\circ  \hat f_q|\|_{L^p(U)}
+\|\sum_q \theta_q d [(\imath\circ f-\imath\circ  \hat f_q)]\|_{L^p(U)}\\
&\le
 C \sum_q \|  \imath\circ f-\imath\circ  \hat f_q\|_{W^{1,p}(\omega'_q\cap U)}\le C\eps.
\end{split}
\end{equation}
The same computation also gives
\begin{equation}\begin{split}
\label{eqlinftyhatphi}
\|d\hat\Phi\|_{L^\infty(U)}
 &=\|
   \sum_q d (\theta_q \imath\circ \hat f_q)\|_{L^\infty(U)}\\
&   \le C \|\imath\circ \hat f_q\|_{L^\infty(U)}
   + C \|d(\imath\circ \hat f_q)\|_{L^\infty(U)}.
\end{split}
\end{equation}
Since $\imath(N)$ is bounded, with \eqref{eq:W2m_extrinsicimathcopy}
we obtain $ \|d\hat\Phi\|_{L^\infty(U)}\le C$.

Analogously,
\begin{equation}\begin{split}
\|\nabla d \hat\Phi\|_{L^m(U)}
 &=\|\sum_q
   \nabla d (\theta_q \imath\circ \hat f_q)\|_{L^m(U)}\\
&\le \sum_q \|
    \nabla d (\theta_q \imath\circ \hat f_q)\|_{L^m(U\cap \omega'_q)}.
\end{split}
\end{equation}
As above, for $\theta:U\to[0,\infty)$ and $\phi:\supp\theta\cap U\to\R$
(which will then be a component of $\imath\circ \hat f_q$)
one obtains, using the Leibniz rule for covariant derivatives,
\begin{equation}
| \nabla d(\theta\phi)|\le
|\nabla d\theta| \,|\phi| +
2 |d\theta|\, |d\phi|+
\theta|\nabla d\phi|.
\end{equation}
Using the properties of $\theta_q$,
and the bounds on $\hat f_q$ and $d\hat f_q$,
\begin{equation}
| \nabla d(\theta_q \imath\circ \hat f_q)|\le
C +
C|\nabla d(\imath\circ \hat f_q)|.
\end{equation}
The $L^m(U\cap \omega'_q)$ norm of the last term was estimated in
\eqref{eq:W2m_extrinsicimathcopy}. Therefore
\begin{equation}\label{eqD2hatPhi}
\begin{split}
\|\nabla d \hat\Phi\|_{L^m(U)}
 &\le C.
\end{split}
\end{equation}

Further, by \eqref{eqdistPhihatLinfty} $\|\Phi-\hat\Phi\|_{L^\infty}\le \delta$.

It remains to show that the projection $\pi$ does not change the estimates.

In order to prove \eqref{eq:mainth:W1p_estimate},
we recall that $\imath\circ \hat f=\pi\circ\hat\Phi$
and $\imath\circ f=\Phi$.
Since $\pi\circ\Phi=\Phi$,
\begin{equation}
d (\pi\circ\hat\Phi)-d\Phi=
 d\pi(\hat\Phi)(d\hat\Phi)-
 d \pi(\Phi)(d\Phi).
\end{equation}
From the fact that $d\pi$ is a smooth function we obtain
\begin{equation}
 |d\pi(\hat\Phi)-d\pi(\Phi)|\le
 C |\hat\Phi-\Phi|.
\end{equation}
Recalling that $|d\Phi|+|d\hat\Phi|\le C$,
\begin{equation}
| d(\pi\circ\hat\Phi)-d\Phi|
\le | d\pi(\Phi)(d\hat\Phi-d\Phi)|+
C |\hat\Phi-\Phi|.
\end{equation}
Since $\pi$ is a projection, $|d\pi(\Phi)|= \sqrt n$. Using
\eqref{eqphihatphiLp} and
\eqref{eqdphihatdphiLp},
\begin{equation}
\| d (\pi\circ\hat\Phi)-d\Phi\|_p
\le C\|d\hat\Phi-d\Phi\|_p+
C \|\hat\Phi-\Phi\|_p\le C\eps
\end{equation}
which concludes the proof of \eqref{eq:mainth:W1p_estimate}.

Similarly, $|d(\imath\circ \hat f)|=|d(\pi\circ \hat\Phi)|\le C |d\hat\Phi|$ and \eqref{eqlinftyhatphi} prove the first bound in \eqref{eq:mainth:W2m_extrinsic}.

It remains to estimate $\|\nabla d (\imath\circ \hat f)\|_{L^m(U)}=\|\nabla d (\pi\circ \hat \Phi)\|_{L^m(U)}$. We recall that
\begin{equation}
\begin{split}
d(\pi\circ\hat\Phi)
 &=(d\pi)(\hat\Phi) d\hat\Phi
 : U \to \Lin(TM, \R^d),
\end{split}
\end{equation}
 each component is a map from $U$ to $T^*M$.
Since $\pi$ is a smooth map from a subset of $\R^d$ to $\R^d$, and $\hat\Phi$ maps $U\subset M$ into
$(\imath(N))_\delta\subset\R^d$,
one easily sees that
(for example taking components and working in local coordinates)
\begin{equation}\label{eqnabladpihatphi}
| \nabla d (\pi\circ \hat\Phi)|\le
 |D^2\pi| |d\hat\Phi|^2
 + |D\pi| |\nabla d\hat\Phi|.
\end{equation}
We recall that $|d\hat\Phi|_{L^\infty}\le C$, that all derivatives of $\pi$ are bounded, so that
\begin{equation}
\| \nabla d (\pi\circ \hat\Phi)\|_{L^m(U)}\le
 C + C \|\nabla d\hat\Phi\|_{L^m(U)}.
\end{equation}
We use \eqref{eqD2hatPhi} and conclude.
\end{enumerate}
\end{proof}

\newcommand{\gdOmega}[1]{\Omega(#1)}
\newcommand{\gdAverage}[3]{(#1)_{#2,#3}}
\def\gdhatF{\widetilde F}
\def\gdhatf{\widetilde f}
\def\Omegaast{\gdOmega{R_\ast}}
\def\Rconcat{{R_{N}}}     \def\Rlarge{{R_\mathrm{chart}}} \def\Rexistence{{R_\mathrm{exist}}}   \def\RextensionG{{R_\mathrm{extg}}}   \def\RFixedpoint{{\tilde R_\mathrm{fxpt}}}        \newcommand{\RFixedpointr}{r}        
\def\RFixedpointnew{{R_\mathrm{fxpt}}}        \def\Restimate{{R_\mathrm{estim}}}         \newcommand{\epsilonloc}[1]{\varepsilon_0}
\newcommand\deltaN{\delta_N}

\begin{proof}[Proof of Proposition~\ref{prop:W2m_local}]
We fix an open set $\omega\subset M$ with $q\in\omega$
and a chart $\varphi:\omega\to\R^n$
with $\varphi(q)=0$.
Possibly making $m$ larger, we can assume that
\begin{equation}\label{eqboundsmp}
m>2n, \qquad m\ge p, \qquad  \frac1m+\frac1p<1.
\end{equation}
In this proof, the average of a function $F$ on a ball $B(y,r)$ is denoted by $\gdAverage{F}{y}{r}$.
Let
\begin{equation}\label{eqpropdefeps}
 \epsilonloc{\omega\cap U}
 :=  \|\dist(df,SO(M,N))\|_{L^p(\omega\cap U)}.
\end{equation}

If $q\in U$ we fix $\Rlarge>0$ with $B_{\Rlarge}^{(\R^n)}(0)\subset\varphi(\omega\cap U)$ and we set $\gdOmega{R}=B_R^{(\R^n)}(0)$, $R\in (0,\Rlarge]$. If $q\in\partial U$, recalling \eqref{eqdeflipschitzboundary}, there exists a function $\gamma\in\Lip(\R^{n-1})$ with $\gamma(0)=0$ and an $\Rlarge>0$ such that
\begin{equation}
 \gdOmega{R}:=
 \varphi(\omega\cap U)\cap B_R^{(\R^n)}(0)=
 \{x\in B_R^{(\R^n)}(0): x_n< \gamma(x')\}, \quad R\in (0,\Rlarge].
\end{equation}

The assertion follows from the following local approximation result in $\R^n$.
There exists $\deltaN>0$, $\Restimate>0$ and $C>0$ (which depend on $q$, $O$, $O'$, $\varphi$, $\psi$) with the following property: For $\delta\in (0,\deltaN]$ there exists an  
$R_\ast\in (0,\Restimate]$ such that for all $f$ as in the statement of the proposition, and  $F:=\psi \circ f \circ \phi^{-1}:\gdOmega{R_\ast}\to \R^n$ there exists a function $\hat F:\gdOmega{R_\ast}\to \R^n$ with
\begin{eqnarray}
\| \hat F - F\|_{W^{1,p}(\Omegaast)} &\le &C  \epsilonloc{\omega\cap U}  ,  \label{eq:W1p_estimate}
\\
 \|\hat F-F\|_{L^\infty(\Omegaast)} &\le &  \delta,
\label{eq:Linfty_extrinsic}
\end{eqnarray}
and
\begin{subequations}\label{eq:W2m_extrinsic}
 \begin{align}
\| \Enabla \hat F\|_{L^\infty(\Omegaast)} & \leq C, \label{eq:W2m_extrinsica} \\
\| \Enabla^2\hat F\|_{L^m(\Omegaast)}
& \le C. \label{eq:W2m_extrinsicb}
 \end{align}
\end{subequations}
The parameter $\deltaN$ is defined in Step~1, 
the radius $\Restimate$ at the end of Step~4 below and the constant $C$ is uniform in $f$ and $R_\ast$.

In Steps~1-3 we construct a function $\gdhatF$ which satisfies these estimates with the exception of the $W^{2,m}$ estimate~\eqref{eq:W2m_extrinsicb}, in Step~4 we show how to obtain $\hat F$ from $\gdhatF$, and we conclude the proof of the proposition.

\bigskip

\textit{Step 1:
There exist $\deltaN>0$, $\Rconcat>0$ and an open set
$B^N\subset\subset\psi(O)$
(which depend on $q$, $O$, $O'$, $\Lambda$, $\varphi$, $\psi$) such
that for any
 $Z\in L^\infty(\gdOmega{\Rlarge};\R^n)$ with
$\| Z \|_\infty \leq \deltaN$
one has
$(F+Z)(\Omega(\Rconcat))\subset B^N$.
}

We first fix $\deltaN$. The choice of $\deltaN$ needs to guarantee that for $Z\in L^\infty(\gdOmega{\Rlarge};\R^n)$ with
$\| Z \|_\infty \leq \deltaN$ the function $\gdhatF = F + Z$ has values in the range of $\psi$ on $O$ and that the function $\gdhatf = \psi^{-1}\circ \gdhatF\circ \phi:\phi^{-1}(\gdOmega{\Rlarge})\subset U\to N$ is well-defined. Since $\Lip(f)\le \Lambda$ on $\omega\cap\overline U$, we have $\Lip(F)\le C\Lambda$ on $\Omega(\Rlarge)$, so that \eqref{eq:Linfty_extrinsic} implies
$\gdhatF(z)\in  B_{\delta + C \Lambda \Rlarge}(F(0))$ on $\Omega(\Rlarge)$. Since $\psi^{-1}(F(0))=f(q)\in O'\subset\subset O$ we can
choose $\deltaN$ small enough such that
$$B^N:=
\overline B_{(1+C\Lambda )\deltaN}(F(0))\subset\subset\psi(O).$$
Set $\Rconcat:=\min\{\deltaN,\Rlarge\}$. With this choice, $\gdhatF(\gdOmega{\Rconcat})\subset B^N\subset \psi(O)$.

In the next steps 
we shall prove the approximation result for $\Restimate\in (0,\Rconcat]$ small enough. The $W^{1,p}$ estimate~\eqref{eq:W1p_estimate} and the $L^\infty$ estimate~\eqref{eq:Linfty_extrinsic} for the approximation error as well as the Lipschitz bound \eqref{eq:W2m_extrinsica} follow from Step~2, the regularity statement \eqref{eq:W2m_extrinsicb} for the approximation from Steps~3 and~4. 

\bigskip

\textit{Step~2:
There exists an $\Rexistence\in (0,\Rconcat]$ such that, for any $F:\gdOmega{\Rexistence}\to\R^n$ as in Step~1, there exists an $\gdhatF\in W^{1,\infty}(\gdOmega{\Rexistence}; \R^n)$ which is the coordinate representation of a harmonic map $\gdhatf:\phi^{-1}(\gdOmega{\Rexistence})\to N$, obeys the
$W^{1,p}$ estimate~\eqref{eq:W1p_estimate} and the $L^\infty$ estimate~\eqref{eq:Linfty_extrinsic} for the approximation error, and
the Lipschitz bound in~\eqref{eq:W2m_extrinsica}.}

The key idea for this construction is that the
function $f$ obeys the weak Piola identity
 (see \Cref{prop:weakpiola} for the extrinsic version,
 \eqref{eqdeltagprelim1} for the intrinsic one)
 with a right-hand side whose $L^p$ norm is controlled by $C\epsilonloc{}$.
 We shall use this in coordinates.
Precisely,
by Proposition~\ref{proppiolacoordinates}
(or by~\eqref{eqdeltagprelim1})
the function
$F:=\psi\circ f\circ\varphi^{-1}:\gdOmega{\Rconcat}
\to\R^n$  obeys
\begin{equation}\label{eqLgFk}
-L_gF_k(x)
 =\hat \Gamma_k(x,F(x),\Enabla F(x))+\partial_i  H_{ki}(x)+ H_k'(x),
\end{equation}
for $k=1, \dots, n$.
Here $\hat \Gamma(\cdot, \cdot, \cdot)\in \R^n$ is the coordinate representation of the quadratic term in the harmonic map equation ,
which is quadratic in the last argument, in the sense that
(see
Proposition~\ref{proppiolacoordinates})
\begin{equation}\label{eqhatGammaQ}
 \hat\Gamma_k(x,a,A)=Q_k(x,a)(A,A),
 \hskip5mm
 k\in\{1,\dots, n\},\,
 a\in \psi(O)\in \R^n,\, A\subset\R^{n\times n},
\end{equation}
and $L_g$ is the Laplace-Beltrami operator 
acting in the sense of distributions on Sobolev functions $u:M\to \R$,
\begin{equation}\label{eqdefLgfromg}
 L_g u(x):=
\frac1{\sqrt {\det g(x)}}\partial_i (\sqrt{\det g(x)} g^{ij}(x) \partial_j  u(x))
\end{equation}
with 
 $g(\cdot)\in\R^{n\times n}$ the representation of the metric of $M$,
see~\eqref{eqdefLgfromgbb}.
The right-hand side obeys the estimate
\begin{equation}\label{eqHheps}
 \|H\|_{L^p(\gdOmega{\Rconcat})} +
 \|H'\|_{L^p(\gdOmega{\Rconcat})} \le C \epsilonloc{\omega\cap U}.
\end{equation}
We look for a radius $\Rexistence\in (0,\Rconcat]$ (independent of $f$) and a map
$\gdhatF:\gdOmega{\Rexistence}\to\R^n$ such that
\begin{equation}\label{eqhatfharmonicmap}
-L_g \gdhatF_k(x) =
\hat \Gamma_k(x, \gdhatF(x), \Enabla \gdhatF(x))\quad \text{ in }\Omega(\Rexistence),
\end{equation}
with estimates which are in particular uniform in the local Lipschitz constants of $\partial U$. Therefore we seek $\gdhatF$ as a correction of $F$ via $\gdhatF = F + Z$ where $Z$ is a solution of a related equation on $B_{\Rexistence}$. If $q\in \partial M$, this approach requires in particular an extension of the metric $g$.

In this case, one first extends the metric to a smooth map defined on $B_{\RextensionG}$ with $\RextensionG\in (0,\Rconcat]$ which fulfills the bounds and the ellipticity conditions in~\eqref{eqgC0} (possibly increasing $C_0$). This gives an extension of the operator $L_g$ which depends on the metric. If $q\in \partial U\setminus \partial M$ or if $q\in U$, choose $\RextensionG$ with $B(q,\RextensionG)\subset M$.
We set $Q(x,\cdot)=0$ for $x\in B_{\RextensionG}\setminus \gdOmega{\RextensionG}$, and obtain that $Q$ satisfies
\begin{align}\label{eqboundsQ}
\begin{aligned}
\|Q\|_{L^\infty(B_\RextensionG\times \psi(O))} & \le C, \\  \|\Enabla_x Q\|_{L^\infty(\Omega(\RextensionG)\times \psi(O))}+
\|\Enabla_a Q\|_{L^\infty(B_{\RextensionG}\times \psi(O))} & \le C.
 \end{aligned}
\end{align}
Finally we extend $H$ and $H'$ by zero to all of $B_{\RextensionG}$.

In order to obtain the equation for $Z$ on $B_{\Rexistence}$, $\Rexistence\in (0,\RextensionG]$, we take
the difference  between \eqref{eqhatfharmonicmap} and  \eqref{eqLgFk}, and obtain on $B_{\RextensionG}$
\begin{equation}\label{eqforZGamma}
-L_g Z_k(x) =
\hat \Gamma_k(x, \gdhatF(x), \Enabla \gdhatF(x))
-\hat \Gamma_k(x,F(x),\Enabla F(x))-\partial_i  H_{ki}(x)- H_k'(x).
\end{equation}
The right hand side
$ \Xi(x,a,A):B_{\RextensionG}\times \psi(O) \times \R^{n\times n}$ is then given by

\begin{equation}\label{eqdefXi}
\begin{split}
 \Xi(x,a,A):=&\hat\Gamma(x,F(x)+a, \Enabla F(x)+A)-\hat\Gamma(x,F(x), \Enabla F(x))\\
 =& Q(x,F(x)+a)(A,A)
 + 2 Q(x,F(x)+a)(A,\Enabla F(x))\\
 &
 + [ Q(x,F(x)+a)-Q(x,F(x))](\Enabla F(x),\Enabla F(x)).
\end{split}\end{equation}

At this point we use the
fixed-point argument in Proposition~\ref{prop:harmonic_replacement_interior}  below. We start by checking the assumptions.
Recalling that $|\Enabla F|\le C\Lambda$ on $\Omega(\RextensionG)$,
a simple computation shows that
\eqref{eqboundsQ} and \eqref{eqdefXi} imply \eqref{eqHaAbBlip} on $B_{\RextensionG}\times \psi(O)\times \R^{n\times n}$, with $C_0$ depending on $\Lambda$.
By~\eqref{eqhhprimeprelim2} and regularity of the coordinate representation,
\begin{equation}\label{eqhhprimeprelim3b}
 \|H\|_{L^\infty(\omega\cap U)} +
 \|H'\|_{L^\infty(\omega\cap U)} \le C_\Lambda,
\end{equation}
so that we can choose $\Lambda_*>0$ such that \eqref{eqassHHstrich} is satisfied for all choices of $F$, $H$, $H'$, and all $r\in(0,\RextensionG]$. We then choose $r_0\in (0,\RextensionG]$ such that $r_0\Lambda_*\subset B^N$, and all assumptions of 
Proposition~\ref{prop:harmonic_replacement_interior} are satisfied.

By  Proposition~\ref{prop:harmonic_replacement_interior} there exists an $\RFixedpoint\in (0,\RextensionG ]$ such that
for any $\RFixedpointr\in (0, \RFixedpoint]$
 equation \eqref{eqforZGamma} has
a unique solution $Z\in W^{1,m}_0(B_{\RFixedpointr};\R^n)\cap \{\|\Enabla Z\|_{L^m}\le 1\}$. By the above argument, $\gdhatF:=Z+F:\gdOmega{\RFixedpointr}\to \R^n$ obeys the harmonic map equation \eqref{eqhatfharmonicmap} in $\gdOmega{\RFixedpointr}$; from \eqref{eqHheps}, \eqref{eqLpboundnablaZh} with $\ell=p$, and Poincar\'e one obtains the $W^{1,p}$ estimate~\eqref{eq:W1p_estimate} for the approximation error $Z=F-\gdhatF$,
\begin{equation}\label{eqDFDhatFclose}
\frac{1}{\RFixedpointr} \|F-\gdhatF\|_{L^{p}(\gdOmega{\RFixedpointr})}+
 \|\Enabla F-\Enabla \gdhatF\|_{L^{p}(\gdOmega{\RFixedpointr})}
 \le C \epsilonloc{\omega\cap U}.
\end{equation}
Further, by \eqref{eqlpboundlinfty} $\Lip(Z;B_{\RFixedpointr/2})\le C\Lambda_*$, therefore~\eqref{eq:W2m_extrinsica} holds,
\begin{equation}\label{eqhatFlip}
 \Lip(\gdhatF,\Omega\cap B_{\RFixedpointr/2})\le  C\Lambda_*.
\end{equation}
Finally, \eqref{eqZLinftyE} guarantees the $L^{\infty}$ estimate~\eqref{eq:Linfty_extrinsic} for the approximation error $Z=F-\gdhatF$ if we choose $\RFixedpointnew\in (0,\RFixedpoint]$ small enough.

This defines $\Rexistence := \RFixedpointnew/2$ and $\gdhatF$ in the assertion of this step.

\bigskip

\textit{Step~3: Interior estimates. 
$\gdhatF\in W^{2,1}_\loc(\gdOmega{\Rexistence})$ and
there exists a constant $C_s>0$ such that for all $B_{\rho}(y)\subset \gdOmega{\Rexistence}$ and all $s\in (1,\infty)$,
 \begin{equation}
\label{eqnablathatfbef}
 \| \Enabla^2\gdhatF \|_{L^\infty(B_{\rho/2}(y))}   \le \frac{C_s}{\rho}\rho^{-n/s}
 \|\Enabla \gdhatF-\gdAverage{\Enabla \gdhatF}{y}{\rho}\|_{L^s(B_{\rho}(y))}+C(s,\Lambda_*).
\end{equation}
}

We apply the elliptic regularity statements contained in  Section~\ref{sec:elliptic}, see in particular Lemma~\ref{le:app_W2m_in_unit_ball} with $b=0$
and Lemma~\ref{eq:app_second_derivative_oscillation}, to obtain a pointwise estimate for $\Enabla^2 \gdhatF$.

To be precise, let $B_{\Rexistence}$ be the ball obtained in Step~2 and $B_\rho(y)\subset B_{\Rexistence}$.
To prove~\eqref{eqnablathatfbef},
we apply Lemma~\ref{eq:app_second_derivative_oscillation} to the function
$u:=\gdhatF-\gdhatF(y):B_{\rho}(y)\to\R^n$.
Equation~\eqref{eq:A_system_quadratic_growth} with
$a^{\alpha\beta}:=\sqrt{\det g}g^{\alpha\beta}$ and, in the notation of \eqref{eqhatGammaQ},
\begin{equation}
G^k(x,a,A):=\sqrt{\det g(x)} Q_k(x,\gdhatF(y)+a)(A,A)
\end{equation}
follows from~\eqref{eqhatfharmonicmap}
(see~\eqref{eqLgufinalpart1}-\eqref{eqLgufinalpart3} below for a similar computation).
Conditions
 \eqref{eq:app_ellipticity_a}-\eqref{eq:app_C2_a}
 follow from the regularity of $g$, with bounds that do not depend on $x$ and $\Rexistence$.
It remains to check
\eqref{eq:app_growth_G}-\eqref{eq:app_DAG}.
We first take a cutoff function $\eta\in C^\infty_c(\psi(O);[0,1])$ with $\eta=1$ on $B^N$,
and replace $Q(x,a)$ by $\eta(a)Q(x,a)$, which is defined on $B_\Rlarge\times \R^n$.
Since the quadratic form $Q$ is bounded, \eqref{eq:app_growth_G} and \eqref{eq:app_DAG} are immediate.
Conditions~\eqref{eq:app_DxG} and~\eqref{eq:app_DzG} follow from the last of~\eqref{eqboundsQ}.

Recalling~\eqref{eqhatFlip}, conditions~\eqref{eq:app_smallness_w1infty} and~\eqref{eq:u0in0} are immediate, with $\Lambda':=C\Rexistence\Lambda_*$.
We set $A_0:=\gdAverage{D\gdhatF}{y}{\rho}$, so that
\eqref{eqbda0rlstr} follows immediately from
\eqref{eq:app_smallness_w1infty}. Then  \eqref{eq:app_interior_C2_normal}
implies~\eqref{eqnablathatfbef}.

\bigskip

\textit{Step~4: Existence of $\Restimate$, $\hat F$, and proof of the global estimates \eqref{eq:W1p_estimate}, \eqref{eq:Linfty_extrinsic}, and~\eqref{eq:W2m_extrinsic}. }

We need to convert the local estimate for $\gdhatF$ in Step~3 into a global one of the form $ |D^2\gdhatF| \le A/\dist(\cdot,\partial\gdOmega{\Rexistence}) + B$ so that we can use the weighted estimates in Section~\ref{se:smith}. By construction, the representations of the metrics $g$ and $g_N$ and their inverses $g^{-1}$ and $g_N^{-1}$ are bounded on $\gdOmega{\Rexistence}$ and $\psi(O)$, respectively, by a constant $M$. The Euclidean rigidity estimate in Lemma~\ref{lemmaFJMwithprefactor} below implies for $X_0$, $Y_0\in \R^{n\times n}$ with $|X_0|$, $|Y_0|\leq M$ that
\begin{equation}\label{eqFJMDhatFgddoppel}
 \|D\gdhatF-\gdAverage{D\gdhatF}{y_0}{\rho}\|_{L^s(B_\rho(y_0))}\le C(M)\|\dist(X_0\, D\gdhatF\,Y_0, SO(n))\|_{L^s(B_{\rho}(y))}.
\end{equation}
To estimate the right hand side, we relate the integrand to $\dist(df,SO(M,N))$.
By the definition in \eqref{eq:defSOMN}, we have
\begin{equation}\label{eqdistDFdfSOn}
 \dist(g_N^{1/2}(F(x)) D F(x) g^{-1/2}(x), SO(n))\le C
  \dist(df\circ \varphi^{-1}(x), SO(M,N)).
\end{equation}
We define $X$, $Y:\gdOmega{\Rexistence}\to\R^{n\times n}_+$ by  $X:=g_N^{1/2}\circ F$ and $Y:=g^{-1/2}$. Since $F$ is Lipschitz and the metrics are smooth, $X$ and $Y$ are Lipschitz with uniform bound $C$. The definition of $\epsilonloc{\omega\cap U}$ in  \eqref{eqpropdefeps} and the estimate \eqref{eqDFDhatFclose} allow us to replace $F$ by $\gdhatF$ and to obtain that
\begin{equation}\label{eqdistDhatFSOn}
\Fdist(x):=\dist(X(x) (D\gdhatF(x)) Y(x),SO(n))
\text{ obeys }
\|\Fdist\|_{L^p(\gdOmega{\Rexistence})}\le C \epsilonloc{\omega\cap U}.
\end{equation}
For any ball $B_\rho(y_0)\subset\gdOmega{\Rexistence}$ we find with $X_0=X(y_0)$, $Y_0 = Y(y_0)$ that $|X-X(y_0)|+|Y-Y(y_0)|\le C \rho$ on $B_\rho(y_0)$. Therefore, for any $s\in (1,\infty)$, one has
\begin{equation}
 \|\dist( X(y_0) (D\gdhatF) Y(y_0),SO(n))\|_{L^s(B_\rho(y_0))} \le
 \|\Fdist\|_{L^s(B_\rho(y_0))} + C \rho^{1+n/s}.
\end{equation}
Together with \eqref{eqFJMDhatFgddoppel} we conclude
\begin{equation}\label{eqFJMDhatF}
 \|D\gdhatF-\gdAverage{D\gdhatF}{y_0}{\rho}\|_{L^s(B_\rho(y_0))} \le
C \|\Fdist\|_{L^s(B_\rho(y_0))} + C \rho^{1+n/s}.
\end{equation}

We finally obtain from the local estimate a global one, using the singular integral estimates presented in Section~\ref{se:smith}.
The key point is the $W^{2,m}$ estimate.

\newcommand{\oldqexponent}{s}
Fix an exponent $\oldqexponent\in(1,p)$, and consider the $\oldqexponent$-maximal function of $\Fdist$, setting for $x\in\R^n$
\begin{equation}
 \Fdistmax(x):= \sup_{\rho>0}\left[ \frac{1}{\rho^n}
 \int_{B_\rho(x)} \Fdist^\oldqexponent\chi_{\gdOmega{\Rexistence}} dy \right]^{1/\oldqexponent} .
\end{equation}
Since $\oldqexponent<p$, by singular-integral estimates
\cite[Th. 1(c), page 5]{Stein1970}
in $L^{p/\oldqexponent}(\R^n)$ applied to the function
$P^\oldqexponent\chi_{\gdOmega{\Rexistence}}$
and its maximal function
$Q^\oldqexponent$,
we obtain
\begin{equation}
\| \Fdistmax\|^\oldqexponent_{L^{p}(\R^n)}=
 \| \Fdistmax^\oldqexponent\|_{L^{p/\oldqexponent}(\R^n)}
\le C \|\Fdist^\oldqexponent\chi_{\gdOmega{\Rexistence}}\|_{L^{p/\oldqexponent}(\R^n)}=  C \|\Fdist\|_{L^{p}(\gdOmega{\Rexistence})}^\oldqexponent,
\end{equation}
so that
with the $L^p$ bound for $P$ in \eqref{eqdistDhatFSOn} one obtains
\begin{equation}
 \| \Fdistmax\|_{L^{p}(\gdOmega{\Rexistence})}\le C\epsilonloc{\omega\cap U}.
\end{equation}
For any $y\in \gdOmega{\Rexistence/2}$, let $\rho=\rho_y:=\dist(y,\partial\gdOmega{\Rexistence/2})$, and  consider the ball
$B=B_\rho(y)$. Using
\eqref{eqFJMDhatF}
and then the definition of $\Fdistmax$,
\begin{equation}
 \| \Enabla \gdhatF - \gdAverage{\Enabla \gdhatF}{y}{\rho} \|_{L^\oldqexponent(B_{\rho}(y))}\le
 C \| \Fdist\|_{L^\oldqexponent(B)}+ C \rho^{1+n/\oldqexponent}
 \le C\rho^{n/\oldqexponent} \Fdistmax(y)+ C \rho^{1+n/\oldqexponent}.
\end{equation}
Therefore the interior estiamte \eqref{eqnablathatfbef} implies
\begin{equation}\label{eqenabla2ab}
 |\Enabla^2\gdhatF|(y)
 \le \frac{C}{\dist(y,\partial\gdOmega{\Rexistence/2})}\Fdistmax(y) + C_\Lambda.
\end{equation}
We are now almost ready to apply
Proposition~\ref{propLipschitzboundarylocal} to each component of the function $\gdhatF$, with $B=C_\Lambda$ and $A=CQ$. We remark that \eqref{eqenabla2ab}
holds in $\gdOmega{\Rexistence/2}$, and in $\gdOmega{\Rexistence/4}$ we have, in the notation of~\eqref{eqdeflipschitzboundary} (which is essentially the same as~\eqref{eqdefomegargamma}),
$\gamma(y')-y_n \le C \dist(y,\partial(\gdOmega{\Rexistence/2})$.
We know that 
$\gdhatF\in W^{1,\infty}(\gdOmega{\Rexistence/2})\cap W^{2,1}_\mathrm{loc}(\gdOmega{\Rexistence/2})$,
\begin{equation}
 \|\gdhatF\|_{W^{1,m}}
 +\|A\|_{L^m}
 +\|B\|_{L^m}\le C_1,\quad
 \|A\|_{L^p}\le C\eps.
\end{equation}
By
Proposition~\ref{propLipschitzboundarylocal}
there are $\gdhatF_h$, $\gdhatF_l:\Omega'\to\R$ such that $\gdhatF=\gdhatF_h+\gdhatF_l$,
\begin{equation}\label{eqresultprop62}
 \|\gdhatF_h\|_{W^{1,p}(\Omega')}\le C \epsilonloc{\omega\cap U},\quad
 \|\gdhatF_h\|_{W^{1,m}(\Omega')}+
 \|\gdhatF_l\|_{W^{2,m}(\Omega')}\le CC_1.
\end{equation}
Therefore $\Restimate := \Rexistence/4$ and $\hat F = \gdhatF_l$  have
properties \eqref{eq:W1p_estimate}-\eqref{eq:W2m_extrinsic}, and the proof is concluded.
In particular,~\eqref{eq:W2m_extrinsica} follows from 
the last of~\eqref{eqresultprop62} and $m>n$. Further,
from~\eqref{eqDFDhatFclose} and~\eqref{eqresultprop62}
we obtain $
\|F-\hat F\|_{W^{1,p}(\Omega')}\le C\eps$, and since both functions are bounded in $W^{1,\infty}$ 
we conclude $
\|F-\hat F\|_{L^\infty}\le C\eps^\gamma$ for some $\gamma=\gamma(p,n)>0$, so that
~\eqref{eq:Linfty_extrinsic} follows for $\eps$ sufficiently small.
If instead $C\eps^\gamma>\delta$, then setting $\hat F=F(0)$ will do.

We then define $\omega': = \phi^{-1}(\gdOmega{\Restimate})$ and $\gdhatf := \psi^{-1}\circ \gdhatF \circ \phi$. Since the charts are smooth, $\gdhatf$ is as regular as $\gdhatF$.
The estimates
\eqref{eq:W1p_estimateimath}, \eqref{eq:W2m_extrinsicimath} and
\eqref{eq:Linfty_extrinsicimath}
in the statement follow from
\eqref{eq:W1p_estimate}, \eqref{eq:W2m_extrinsic} and
\eqref{eq:Linfty_extrinsic},
and the fact that $\imath\circ\psi^{-1}\in C^\infty(\psi(O);\R^d)$,
with uniform bounds on the compact set $\overline {B^N}$.
\end{proof}

We next prove the Lemma on rigidity with approximately constant metric that was used to obtain \eqref{eqFJMDhatF} above.
\begin{lemma}\label{lemmaFJMwithprefactor}
For any $\kappa\ge1$, $s\in (1,\infty)$, $n\ge 2$ there is $C=C(s,\kappa,n)>0$ such that the following holds:

For any matrices $X$, $Y\in\R^{n\times n}_+$ with $|X|+|X^{-1}|+|Y|+|Y^{-1}|\le \kappa$, any ball $B\subset \R^n$ and any $v\in W^{1,s}(B;\R^n)$ one has
 \begin{equation}\label{eqFJMwithXY}
  \|Dv-(Dv)\|_{L^s(B)}
  \le C\|\dist(X\, Dv\,Y, SO(n))\|_{L^s(B)}.
 \end{equation}
 \end{lemma}
\begin{proof}
 By scaling and translation we can assume $B=B_1(0)$.
Let $w(x):=Xv(Yx)\in W^{1,s}(Y^{-1}B_1;\R^n)$,
so that $Dw(x)=XDv(Yx)Y$ and
\begin{equation}
 \dist(Dw,SO(n))(Y^{-1}y)= \dist(X DvY, SO(n))(y).
\end{equation}
The sets $Y^{-1}B_1$ are uniformly Lipschitz, therefore the Friesecke-James-Müller rigidity estimate~\eqref{eq:euclidean_rigidity} holds uniformly in $Y$
\cite[Th.~5.10]{ContiZwicknagl2023}. This proves \eqref{eqFJMwithXY}.
 \end{proof}

\section{Review of elliptic estimates}\label{sec:elliptic}

In this Section we recall classical regularity estimates for linear elliptic PDEs and for nonlinear second-order equations with quadratic growth in the gradient  \cite{Morrey,GilbargTrudinger01}. For us it is important to obtain conditions on the coefficients so that the estimates are uniform. Therefore we briefly sketch the proofs.
In this section, greek indices have values in $\{1,\ldots,n\}$.

\subsection{Elliptic regularity for linear equations}

For a set
\begin{equation}\label{eqassumptomega}
 \text{ $\Omega \subset \R^n$  open, bounded and connected}
\end{equation}
we consider elliptic PDEs of the form
\begin{equation}  \label{eq:app_scalar_divergence_form}
- D_\beta (a^{\alpha \beta} D_\alpha u) = - D_\alpha f^\alpha + F  \quad \text{in $\Omega$}
\end{equation}
with $a\in C^0(\overline \Omega;\R^{n\times n}_\sym)$ and we  assume that there exists $\lambda>0$, $K>0$ with
\begin{alignat}{2}
a^{\alpha \beta}(x) \xi_\alpha \xi_\beta &\ge \lambda |\xi|^2  \quad &&\text{for all $\xi \in \R^n$ and all $x \in \Omega$,}
\label{eq:app_ellipticity_a}  \\
|a^{\alpha \beta}(x)| & \le  K  \quad &&\text{for all $x \in \Omega$}. \label{eq:app_boundedness_a}
\end{alignat}
Often we will make the stronger assumption that the $a^{\alpha \beta}$ are Lipschitz and
\begin{equation}\label{eq:app_lipschitz_a}
|D a^{\alpha \beta}(x)| \le K_1
\quad \text{for almost every $x \in \Omega$ }
\end{equation}
or that the $a^{\alpha \beta}$ are in $C^{1,1}$ and, in addition to  \eqref{eq:app_lipschitz_a},
\begin{equation}\label{eq:app_C2_a}
|D^2 a^{\alpha \beta}(x)| \le K_2
\quad \text{for almost every $x \in \Omega$.}
\end{equation}

We first recall the existence of $W^{1,p}$ solutions for zero boundary conditions.

\begin{lemma}[Global $W^{1,p}$ regularity]    \label{le:w1p_dirichlet} Let $s, p \in (1,\infty)$ with $\frac1s \le \frac1p + \frac1n$.
Assume \eqref{eqassumptomega} holds, that $\Omega$ has $C^1$ boundary, and that  the coefficients $a\in C^0(\overline \Omega;\R^{n\times n}_\sym)$ satisfy \eqref{eq:app_ellipticity_a}
and  \eqref{eq:app_boundedness_a}. Let $f \in L^p(\Omega; \R^n)$ and $F \in L^s(\Omega)$.
Then there exists a unique $u \in W^{1,p}_0(\Omega)$ such that $u$ is a distributional  solution of  \eqref{eq:app_scalar_divergence_form}. Moreover,
\begin{equation}   \label{eq:app_global_W1p}
\| u\|_{W^{1,p}(\Omega)} \le C \left( \|f\|_{L^p(\Omega)} + \|F\|_{L^s(\Omega)}  \right).
\end{equation}
The constant $C$ depends on $n$, $p$, $s$, $\Omega$, the constants $\lambda$ and $K$ in
\eqref{eq:app_ellipticity_a} and \eqref{eq:app_boundedness_a}, and the modulus of continuity of the $a^{\alpha \beta}$.
\end{lemma}

\begin{proof}
It suffices to treat the case $s\le p$.
We first recall the standard reduction to the case $F = 0$. Indeed, if $F \ne 0$, we can extend $F$ by zero to $\R^n$ and write $F$ as $F = - D_\alpha g^\alpha$ where $g^\alpha = (D_\alpha N) \ast F$ and where $N$ is the fundamental solution of $-\Delta$. The kernel  $D_\alpha N$ is $1-n$ homogeneous and bounded  on the unit sphere. By standard estimates for such kernels  \cite[Thm.\ 1, Paragraph 1, Chapter 5]{Stein1970} we have $\| g^\alpha\|_{L^p(\Omega)} \le C  \|F \|_{L^s(\R^n)} =C  \|F \|_{L^s(\Omega)}$ where the constant depends only on $p$, $s$ and $|\Omega|$.

For $F= 0$, Lemma~\ref{le:w1p_dirichlet} is a special case of \cite[Thm.\ 1]{auscher-qafsaoui2002} where coefficients in the larger space $\mathrm{vmo}(\Omega) \cap L^\infty(\Omega)$ are considered.
\end{proof}

If the right hand side of \eqref{eq:app_scalar_divergence_form} is in $L^p$ and if the coefficients are Lipschitz, then the solution is in $W^{2,p}$ on any open set $\Omega'$ with $\overline{\Omega'} \subset \Omega$.

\begin{lemma}    \label{eq:app_w2q_estimates_first}  Suppose that
 \eqref{eqassumptomega} holds, that $\Omega' \subset\subset \Omega$ is open, and that $a: \Omega \to \R^{n\times n}_\sym$ satisfy \eqref{eq:app_ellipticity_a}-\eqref{eq:app_lipschitz_a}. Let $F \in L^p(\Omega)$, $p\in(1,\infty)$. If $u \in W^{1,p}(\Omega)$ is a distributional solution of
\begin{equation}  \label{eq:app_linear_divergence}
- D_\beta (a^{\alpha \beta} D_\alpha u) = F
\end{equation}
then $u \in W^{2,p}(\Omega')$ and
\begin{equation}  \label{eq:app_w2p_interior}
\| u \|_{W^{2,p}(\Omega')} \le C (\|F\|_{L^p(\Omega)} + \| u \|_{W^{1,p}(\Omega)})
\end{equation}
where $C$ depends on $n$, $p$, $\lambda$, $K$, $K_1$, $\Omega'$, and $\Omega$.
\end{lemma}

\begin{proof} If we assume, in addition, that $u \in W^{2,p}_{\loc}(\Omega)$, then the equation can be rewritten as
$$ -a^{\alpha \beta} D_{\alpha \beta} u - b^{\alpha} D_{\alpha} u =  F,$$
with $b^{\alpha} := D_{\beta} a^{\alpha \beta} \in L^\infty(\Omega)$ and $\| b^{\alpha}\|_{L^\infty} \le n K_1$,
and the  estimate   \eqref{eq:app_w2p_interior} follows from  \cite[Thm.\  9.11]{GilbargTrudinger01}.

To show that  $u \in W^{2,p}_{\loc}(\Omega)$ one can use the usual convolution argument. Let $\Omega'$ be open with  $\overline{\Omega'} \subset \Omega$. Then there exists an open set $\Omega''$ such that $\overline{\Omega'} \subset \Omega'' \subset \overline{\Omega''} \subset \Omega$.
The
function $u_k := \rho_k \ast u$, with $\rho_k\in C_c^\infty(B_{1/k})$ a mollification kernel, for $k$ large enough
is a $C^\infty$ function in $\Omega''$ and satisfies the equation
$$ - D_\beta (a^{\alpha \beta} D_\alpha u_k) = \rho_k \ast F + D_\beta (\rho_k  \ast (a^{\alpha \beta} D_\alpha u)
- a^{\alpha \beta} (\rho_k \ast D_\alpha u)).
$$
The desired assertion $u \in W^{2,p}(\Omega')$ follows from the $W^{2,p}$ estimate for $u_k$ (with $\Omega$ replaced by $\Omega''$) and the commutator estimate
$$  \|  D_\gamma (\rho_k  \ast (a^{\alpha \beta} v)
- a^{\alpha \beta} (\rho_k \ast v)) \|_{L^p(\Omega'')} \le C\,   \Lip(a^{\alpha \beta})  \, \| v\|_{L^p(\Omega)}
$$
which holds for all $v\in L^p(\Omega)$.
\end{proof}

By iteration of  Lemma~\ref{eq:app_w2q_estimates_first} we obtain the following improvement of interior regularity for $W^{1,s}$ solutions of elliptic systems in diagonal form with right hand sides with higher regularity. 

\begin{lemma}   \label{le:app_W2m_in_unit_ball}
Let $\frac12 \le r' < r \le 1$. Suppose that $a: B_r \to \R^{n\times n}_\sym$ satisfies \eqref{eq:app_ellipticity_a}-\eqref{eq:app_lipschitz_a} in $B_r$ and
that $b^{i\alpha}_j \in L^\infty(B_r)$. Let $ s, m \in (1,\infty)$.
Suppose that $w \in W^{1,s}(B_r;\R^d)$ satisfies
\begin{equation}  \label{eq:app_scalar_rhs}
-  D_\beta (a^{\alpha \beta}  D_\alpha w^{i}) + b^{i\alpha}_j D_\alpha w^{j}  = F^{i}
\end{equation}
with $F \in L^m(B_r;\R^d)$.  Then  $w \in W^{2,m}(B_{r'})$ and
\begin{equation}  \label{eq:app_W2m_in_unit_ball}
\| w\|_{W^{2,m}(B_{r'})} \le C (\|F\|_{L^m(B_r)} + \| w\|_{W^{1,s}(B_r)})
\end{equation}
where $C$ may depend on $n$, $s$, $m$,  $\lambda$, $K$, $K_1$, $\|b \|_{L^\infty}$, $r$, and $r'$.
\end{lemma}

\begin{proof}
 The estimate  \eqref{eq:app_W2m_in_unit_ball} follows from
Lemma~\ref{eq:app_w2q_estimates_first}  and a standard bootstrap argument.
To see this, let $r' \le \rho' < \rho \le r$,
$\sigma\in(1,m]$,
and assume that $w \in W^{1,\sigma}(B_\rho;\R^d)$ satisfies  \eqref{eq:app_scalar_rhs}
with $F \in L^m(B_\rho;\R^d)$.  Then, by Lemma~\ref{eq:app_w2q_estimates_first} applied to the equation for $w^{i}$ with right hand side $F^i - b_j^{i\alpha}D_\alpha w^j\in L^\sigma$ we get
$w \in W^{2,\sigma}(B_{\rho'};\R^d)$ and
\begin{equation} \label{eq:app_W2p_estimatesa}
\| w\|_{W^{2,\sigma}(B_{\rho'})} \le C (\|F\|_{L^m(B_\rho)} + \| w\|_{W^{1,\sigma}(B_\rho)}),
\end{equation}
where $C$ depends on $n$, $\sigma$, $m$,  $\lambda$, $K$, $K_1$, $\|b \|_{L^\infty}$, $\rho$, and $\rho'$.

If $m \le s$, the desired assertion follows from  \eqref{eq:app_W2p_estimatesa} with $\sigma = m$, $\rho =r$ and $\rho' = r'$.

If $m > s$,  we
pick an integer $L \ge 1$ such that
$$\frac1s - \frac{L}n  \le \frac1m < \frac1s - \frac{L-1}{n}$$
and then define
$\sigma_l$ for $l = 0, \ldots, L$
by
$$\frac{1}{\sigma_l} = \frac1s - \frac{l}n  \quad \text{for $0 \le l \le L -1$}, \quad \sigma_L = m. $$
Then we have the embedding  $W^{2,\sigma_l} \hookrightarrow W^{1,\sigma_{l+1}}$ for $l \le L-1$
and the assertion  \eqref{eq:app_W2m_in_unit_ball} follows by successively applying
 \eqref{eq:app_W2p_estimatesa} on a sequence of concentric balls.

\end{proof}

 We will also use the following Schauder estimate. 

\begin{lemma}    \label{le:app_linear_schauder}
Let $\Omega \subset \R^n$ be open and bounded, let $\Omega' \subset\subset\Omega$ be open,  fix $\tau \in (0,1)$ and suppose that $u \in C^{1,\tau}(\Omega)$ is a distributional solution of the equation
$$ - D_\beta (a^{\alpha \beta} D_\alpha u) = F$$
where $a: \Omega \to \R^{n\times n}_\sym$ satisfies \eqref{eq:app_ellipticity_a}-\eqref{eq:app_C2_a} and $F \in C^{0,\tau}(\Omega)$. Then $u \in C^{2,\tau}(\Omega')$ and
\begin{equation} \label{eq:app_schauder}
\| u \|_{C^{2,\tau}(\Omega')} \le C( \| u \|_{C^{1,\tau}(\Omega)} + \|F\|_{C^{0,\tau}(\Omega)} )
\end{equation}
where $C$  may depend on $\Omega$, $\Omega'$,  $n$, $\tau$, $\lambda$, $K$, $K_1$ and  $K_2$.
\end{lemma}

\begin{proof}
If $u \in C^{2,\tau}_{\loc}(\Omega)$ then the assertion follows from     \cite[Cor.\ 6.3]{GilbargTrudinger01}. In the general case  one can proceed as in the proof of Lemma~\ref{eq:app_w2q_estimates_first}. One chooses an open set $\Omega''$ with $\overline{\Omega'} \subset \Omega'' \subset \overline{\Omega''} \subset \Omega$ and uses the commutator estimate
$$  \|  D_\gamma (\rho_k  \ast (a^{\alpha \beta} v)
- a^{\alpha \beta} (\rho_k \ast v)) \|_{C^{0,\tau}(\Omega'')} \le C\,  (K_1 + K_2)  \, \| v\|_{C^{0,\tau}(\Omega)}
$$
for $v \in C^{0,\tau}(\Omega)$  and sufficiently large $k$.
\end{proof}

\subsection[Estimates for elliptic systems with quadratic growth in ${Du}$ and diagonal principal part]{Estimates for elliptic systems with quadratic growth in $\boldsymbol{Du}$ and diagonal principal part}

In this section we loosely follow the notation in \cite[Chapter 9.4]{Jost2017}.
Let $\Omega \subset \R^n$ be open.  We study  weak solutions $u: \Omega \to \R^d$ of
\begin{equation}   \label{eq:A_system_quadratic_growth}
-D_\beta(  a^{\alpha \beta}(x) D_\alpha u^{i})(x) = G^{i}(x, u(x), Du(x)),
\end{equation}
$i=1, \dots, d$,
and obtain estimates on balls $B_r=B(x_0,r)\subset\subset \Omega$.
We assume that the maps $(x,z,A) \to G^{i}(x,z,A)$,  defined on $\Omega \times \R^d \times \R^{d \times n}$, are $C^1$ and satisfy the growth conditions
\begin{eqnarray}
|G(x,z,A)| &\le&  c_1 |A|^2 , \quad    \label{eq:app_growth_G} \\
|D_x G(x,z,A)|  &\le& \gamma_1 |A|^2,
\label{eq:app_DxG} \\
|D_z G(x,z,A)| &\le&   \gamma_3 |A|^2,  \label{eq:app_DzG} \\
|D_A G(x,z,A)| &\le& \gamma_5 |A|   \label{eq:app_DAG}
\end{eqnarray}
 for all $(x,z,A) \in \Omega \times \R^d \times \R^{d \times n}$.
 The harmonic map equation has  this structure, with the exception that the nonlinearity $G$ is in general only defined
for $z$ in a suitable ball $B$ in $\R^d$. Multiplying $G$ by $\eta(z)$, where $\eta$ is a cut-off function with $\eta=1$ in $B'\subset\subset B$ we can  put the harmonic
map equation into the setting above.

\begin{lemma}\label{le:C1_C2_estimate} Let $m > 2n$, $C' >0$ and $r \in (0,1]$. Suppose that $a: B_r(x_0) \to \R^{n\times n}_\sym$ satisfies \eqref{eq:app_ellipticity_a}-\eqref{eq:app_lipschitz_a} and that $G$ satisfies  \eqref{eq:app_growth_G}.
Assume  that $u \in W^{1,m}(B_r(x_0); \R^d)$ is
a weak solution of  \eqref{eq:A_system_quadratic_growth} in $B_r(x_0)$ with
\begin{equation}   \label{eq:app_apriori_bound_Du_Lm}
 \| Du\|_{L^m(B_r(x_0))} \le C' r^{-1 + n/m}.
\end{equation}
Then $u \in C^1(B_{3r/4}(x_0);\R^d)$ and
\begin{equation}\label{eq:app_interior_lip}
\| Du \|_{C^0(B_{3r/4}(x_0))} \le C  r^{-n/m} \| Du \|_{L^m(B_{r}(x_0))},
\end{equation}
where
$C$ depends on $n,d,m$, $C'$  as well as  the constants $\lambda$,  $K$, $K_1$,  and  $c_1$ in
\eqref{eq:app_ellipticity_a}-\eqref{eq:app_lipschitz_a} and   \eqref{eq:app_growth_G}.

If, in addition, the coefficients $a^{\alpha \beta}$ satisfy \eqref{eq:app_C2_a} and $G$ satisfies
\eqref{eq:app_DxG}-\eqref{eq:app_DAG}, then
 $u \in C^2(B_{r/2}(x_0); \R^d)$ and
\begin{equation}\label{eq:app_interior_C2}
\| D^2 u \|_{C^0(B_{r/2}(x_0))} \le C r^{-1-n/m} \| Du \|_{L^m(B_{r}(x_0))}. \end{equation}
Here $C$ depends on $n,d,m$, $C'$  as well as  the constants $K$, $K_1$, $K_2$, $c_1$, $\gamma_1$, $\gamma_3$,
and $\gamma_5$
in the estimates
\eqref{eq:app_ellipticity_a}-\eqref{eq:app_C2_a} and
\eqref{eq:app_growth_G}-\eqref{eq:app_DAG}.
\end{lemma}

The proof is similar to the one of Lemma~\ref{eq:app_second_derivative_oscillation} below, but simpler, therefore we only sketch it.

\begin{proof}[Sketch of proof]
It suffices to show the result for $x_0 = 0$, $r=1$, and $u$ with average 0, and then apply it to
$u_r(x) := u(x_0+rx)-\bar u$,
with the same rescaling as in~\eqref{eq:A_system_quadratic_growth_rescaled} below.

The assertion follows from the usual bootstrap argument.
The quadratic growth condition \eqref{eq:app_growth_G} and
\eqref{eq:app_apriori_bound_Du_Lm} imply that
$F:=G(\cdot, u, Du)$ obeys $\|F\|_{L^{m/2}}\le C \|Du\|^2_{L^m}\le C \|Du\|_{L^m}$. From
Lemma~\ref{le:app_W2m_in_unit_ball} we obtain a bound for $u$ in $W^{2,m/2}(B_{3/4})$ and hence in $C^{1,\tau}(B_{3/4})$ with $\tau = 1 - {2n}/{m} > 0$. This proves~\eqref{eq:app_interior_lip}.

An explicit computation (see~\eqref{eq:hoelder_G} below), using \eqref{eq:app_growth_G}-\eqref{eq:app_DAG} and then \eqref{eq:app_apriori_bound_Du_Lm} to eliminate higher-order terms, leads to
$\|F\|_{C^{0,\tau}}\le C\|u\|_{C^{1,\tau}}$; the conclusion then follows from Lemma~\ref{le:app_linear_schauder}.
\end{proof}

\begin{lemma}  \label{eq:app_second_derivative_oscillation} Let $\Lambda' >0$ and $r \in (0,1]$.
Suppose that $a\in C^2( B_r;\R^{n\times n}_\sym)$ satisfies \eqref{eq:app_ellipticity_a}-\eqref{eq:app_C2_a} and that $G$ satisfies \eqref{eq:app_growth_G}-\eqref{eq:app_DAG}.
Suppose further that $u \in W^{1,\infty}(B_r; \R^d)$ is
a weak solution of  \eqref{eq:A_system_quadratic_growth} in $B_r$ and
\begin{equation}   \label{eq:app_smallness_w1infty}
r \| Du \|_{L^\infty(B_r)}  \le \Lambda',
\end{equation}
\begin{equation}\label{eq:u0in0}
u(0) = 0.
\end{equation}
Then $u \in C^2(B_{r/2}; \R^d)$ and, for every
 $s\in(1,\infty)$ and every
$A_0 \in \R^{d \times n}$ with
\begin{equation}\label{eqbda0rlstr}
r |A_0| \le \Lambda',
\end{equation}
we have
\begin{align}
\| Du - A_0 \|_{C^0(B_{3r/4})} \le &  \,  C  r^{-n/s} \| Du  - A_0\|_{L^s(B_{r})}  + Cr( |A_0| + |A_0|^2) ,  \label{eq:app_interior_lipc}  \\
\|D^2 u\|_{C^0(B_{r/2} )} \le   & \, \frac{C}{r} r^{-n/s}  \| Du -A_0 \|_{L^s(B_{r})}
+   C( A_0|   +    |A_0|^2).
 \label{eq:app_interior_C2_normal}
\end{align}
Here $C$ depends on $n$, $d$, $s$, $\Lambda'$, $K$, $K_1$, $K_2$, $c_1$, $\gamma_1$, $\gamma_3$,
and $\gamma_5$. 
\end{lemma}
The proof is also based on a standard bootstrap argument, as for Lemma \ref{le:C1_C2_estimate}; we give some more details here to clarify how the translation by $A_0$ and the scaling $r$ affect the argument.
\begin{proof}
By  \eqref{eq:app_smallness_w1infty} and
\eqref{eq:app_growth_G} we have $F:=G(\cdot, u, Du)\in L^\infty$, and Lemma~\ref{eq:app_w2q_estimates_first} applied componentwise
implies $u \in W^{2,q}_\loc(B_r;\R^d)$ for all $q < \infty$.

We start from the case $r=1$. The function $ w(x) := u(x) - A_0 x$
satisfies the equation
\begin{equation}  \label{eq:app_PDE_w}
 -D_\beta (a^{\alpha \beta} D_\alpha w^{i})(x) = G^{i}(x,u(x),Du(x)) +D_\beta a^{\alpha \beta}(x) (A_0)^{i}_\alpha.
 \end{equation}
 We rewrite it as
\begin{equation}
 \, -D_\beta (a^{\alpha \beta}(x) D_\alpha w^{i})(x)
= \,   b^{i\alpha}_j(x) D_\alpha w^j(x)  +   \hat f^i(x),
\end{equation}
where
$\hat f^i(x):=G^{i}(x,u(x),A_0) +D_\beta a^{\alpha \beta}(x) (A_0)^{i}_\alpha$
and
$$ b^{i\alpha}_j (x):= \int_0^1 \frac{\partial G^{i}}{\partial A^{j}_\alpha}(x,u(x),A_0 + t Dw(x)) \, dt.$$
By~\eqref{eq:app_growth_G} we obtain
$\|\hat f\|_{L^\infty}\le  C( |A_0|^2+\|Da\|_{C^0} |A_0|)$;
by~\eqref{eq:app_DAG}, \eqref{eq:app_smallness_w1infty}, and \eqref{eqbda0rlstr}
we obtain
$\|b^{i\alpha}_j(x)\|_{L^\infty} \le  C $.
Let $m:=2n$ and $ \tau := 1 - \frac{n}{m} =1/2$. Then  Lemma~\ref{le:app_W2m_in_unit_ball}, using
Poincar\'e's estimate and~\eqref{eq:u0in0}, implies
\begin{equation}  \label{eq:app_C1tau_w}
\|w\|_{C^{1,\tau}(B_{3/4})} \le C ( \|Du - A_0\|_{L^s(B_1)} + \|Da\|_{C^0} |A_0| +  |A_0|^2).
\end{equation}
With \eqref{eq:app_lipschitz_a} and
$Du-A_0=Dw$ we obtain \eqref{eq:app_interior_lipc} for $r=1$.

In order estimate the second derivatives
we shall use Schauder theory (Lemma~\ref{le:app_linear_schauder})
for~\eqref{eq:app_PDE_w}, viewed as linear elliptic PDE for $w$.
From $|Du|^2\le 2|A_0|^2+2|Dw|^2$,
\eqref{eq:app_smallness_w1infty}, \eqref{eqbda0rlstr}
and \eqref{eq:app_C1tau_w} we have
\begin{equation}\label{eqduc02}
 \|Du\|_{C^0(B_{3/4})}^2  \le  C (\|Du - A_0\|_{L^s(B_1)} + |A_0|^2),
\end{equation}
so that
\begin{equation}  \label{eq:app_rhs_w_linfty}
 \|G(\cdot, u, Du)\|_{C^0(B_{3/4})}  \le c_1 \|Du\|_{C^0(B_{3/4})}^2
 \le  C (\|Du - A_0\|_{L^s(B_1)} + |A_0|^2).
 \end{equation}
 We write
\begin{equation}\label{eq:hoelder_G}
\begin{split}
&  |G(x,u(x), Du(x)) - G(y,u(y),Du(y))| \\
\le & \|D_xG\|_\infty |x-y|+
\|D_zG\|_\infty |u(x)-u(y)|
+\|D_AG\|_\infty |Du(x)-Du(y)|
 \end{split}  \end{equation}
and estimate
\begin{align}& \,  [G(\cdot, u(\cdot), Du(\cdot)]_{C^{0,\tau}(B_{3/4})}   \nonumber \\
\le & \,  \gamma_1 \|Du\|_{C^0(B_{3/4})}^2
+ \gamma_3  \|Du\|_{C^0(B_{3/4})}^3 + \gamma_5   \|Du\|_{C^0(B_{3/4})} [Du]_{C^{0,\tau}(B_{3/4})}  \nonumber\\
\le & \,C \|Du\|_{C^0(B_{3/4})}^2  + C  [Du]_{C^{0,\tau}(B_{3/4})}
\nonumber  \\
\le & \,  C (\|Du - A_0\|_{L^s(B_1)}+  |A_0|^2 +\|Da\|_{C^0} |A_0|),
\label{eq:app_rhs_w_schauder}
\end{align}
where in the last step we used
\eqref{eq:app_C1tau_w} and~\eqref{eqduc02}.
Similarly,
\begin{equation} \label{eq:app_coefficients_schauder}
   \|D_\beta a^{\alpha \beta} (A_0)^{i}_\alpha\|_{C^{0,\tau}} \le C
   \|Da\|_{C^1}|A_0|.
   \end{equation}
Lemma~\ref{le:app_linear_schauder} then gives the
 bound
 \begin{equation}\label{eqD2wCt}
 \| w\|_{C^{2,\tau}(B_{1/2})} \le C (  \|Du - A_0\|_{L^s(B_1)}  + |A_0|^2 + \|Da\|_{C^1} |A_0|).
 \end{equation}
Using \eqref{eq:app_lipschitz_a}, \eqref{eq:app_C2_a}
 and
$Du-A_0=Dw$ we obtain \eqref{eq:app_interior_C2_normal} for $r=1$.

It remains to consider the case $r \in (0,1)$. Define $u_r: B_1 \to \R^d$ by
\begin{equation*}u_r(x) := u(rx),
\quad a^{\alpha \beta}_r(x) := a^{\alpha \beta}(rx), \quad G^{i}_r(x,z,A) := r^2 G^{i}(rx, z, \frac1r A).
\end{equation*}
 Then $D u_r(x) = r Du(rx)$ and thus $u_r$ is a weak solution of
 \begin{equation}\label{eq:A_system_quadratic_growth_rescaled}
-D_\beta(  a_r^{\alpha \beta}(x) D_\alpha u_r^{i})(x) = G^{i}_r(x, u_r(x), Du_r(x)).
\end{equation}
The coefficients $a^{\alpha \beta}_r$ and the functions $G_r$ satisfy the conditions
\eqref{eq:app_ellipticity_a}-\eqref{eq:app_lipschitz_a} and \eqref{eq:app_growth_G}-\eqref{eq:app_DAG},
with the same values of  $\lambda$, $K$, $K_1$, $K_2$, $c_1$, $\gamma_1$, $\gamma_3$, $\gamma_5$.

 Application of  \eqref{eq:app_C1tau_w} and \eqref{eqD2wCt}
  with $A_0$ replaced by $r A_0$ gives
 \begin{equation*}\begin{split}
 \| Du_r - r A_0\|_{C^0(B_{3/4})}  &\le   C \| Du_r - r A_0 \|_{L^s(B_{1})}
 +  \, C(  \|Da_r\|_{C^0} |A_0| r +   |A_0|^2 r^2), \\
 \|D^2 u_r\|_{C^0(B_{1/2} )}  & \le      C \| Du_r - r A_0 \|_{L^s(B_{1})}
 +  \, C(  \|Da_r\|_{C^1} |A_0| r +   |A_0|^2 r^2).
 \end{split}\end{equation*}
It suffices to insert $  \| Du_r - r A_0 \|_{L^s(B_{1})} = r^{1-n/s} \| Du - A_0\|_{L^s(B_r)}$,
 $Du_r(x) = r Du(rx)$ and $D^2 u_r(x) = r^2 Du(r x)$, and the same for $a_r$, to conclude.
\end{proof}

\section{Construction of a local harmonic approximation}\label{sec:solharmmap}
We show here that, if $r$ is sufficiently small, 
equation \eqref{eqforZGamma}, which is
\begin{equation}\label{eq:harmonic_extension_interiortwo}
-L_g Z_k(x) =
\Xi_k(x, Z(x), \Enabla Z(x))
-\partial_i  H_{ki}(x)- H'_k(x),
\end{equation}
with $\Xi$ as in \eqref{eqdefXi},
has a solution $Z\in W^{1,m}_0(B_r;\R^n)$
which can be bounded in terms of $H$ and $H'$.
\begin{proposition}   \label{prop:harmonic_replacement_interior} Let   $m \in (2n, \infty)$, $p \in (1,\infty)$ with $ m \ge p$
and  $\frac1m + \frac1p < 1$.
Let $r_0>0$, $C_0>0$, $\Lambda\ge 1$.
Then  there exist $r_1\in (0, r_0]$ and $C_1 > 0$ with the following property.

Let $g\in C^\infty(B_{r_0};\R^{n\times n}_{+})$
and assume that
$\Xi:B_{r_0}\times B_{r_0\Lambda}\times \R^{n\times n}\to\R^n$
is measurable and obeys
\begin{equation}\label{eqHaAbBlip}
 \begin{split}
|\Xi(x,a,A)-\Xi(x,b,B)|\le&
  C_0 |a-b|(1+|A|^2+|B|^2)\\
  &
 +C_0 |A-B|(1+|A|+|B|),\\
 |\Xi(x,a,0)|\le & C_0|a|
\end{split}\end{equation}
for all $x, a,b,A,B$, and
\begin{equation}\label{eqgC0}
 \frac1{C_0}\Id \le g(x)\le C_0 \Id, \hskip5mm
 |Dg|(x)+|D^2g|(x)\le C_0
\end{equation}
for all $x$.
If for some $r\in (0, r_1]$, $\Herror: B_{r}\to\R^{n\times n}$, $\Hprimeerror: B_{r}\to\R^{n}$ obey
\begin{equation}\label{eqassHHstrich}
\| H\|_{L^m(B_r)}+r \| H'\|_{L^m(B_r)}  \le \Lambda r^{n/m},
\end{equation}
then the  equation~\eqref{eq:harmonic_extension_interiortwo} has a unique solution in
\begin{equation}\label{eqdefEr}
E_r:=\{Z \in  W_0^{1,m}(B_r;\R^n)  : \| \Enabla Z\|_{L^m(B_r)} \le 1\}.
\end{equation}
This solution satisfies
\begin{equation}\label{eqlpboundlinfty}
\|\Enabla Z\|_{L^\infty(B_{r/2})}\le C_1\Lambda
\end{equation}
and
\begin{equation}\label{eqLpboundnablaZh}
\|\Enabla Z\|_{L^{\ell}{(B_r)}} \le C_1  \|\Herror\|_{L^\ell(B_r)} +  C_1 r \|\Hprimeerror\|_{L^\ell(B_r)}  \quad \text{for all $\ell \in [p,m]$}.
\end{equation}
\end{proposition}

\begin{proof} 
\newcommand\Enorm[1]{\|#1\|_{E_r}}
\emph{Step 1: Existence of $Z$.}\\
For the purpose of this proof we set
\begin{equation}\label{eqdefEnorm}
\Enorm{Z}:= \| \Enabla Z\|_{L^m(B_r)} \quad \text{for $Z \in W^{1,m}_0(B_r;\R^n)$.}
\end{equation}
We show existence (for sufficiently small $r$)  by an application of the Banach fixed point theorem
to the operator $T$ defined by
$$ T(Z) := (-L_g)^{-1}    \left[
\Xi(\cdot,Z, \Enabla Z)
-\Div H - H' \right]$$
which is diagonal, in the sense that
$$ T_k(Z) = (-L_g)^{-1}    \left[
\Xi_k(\cdot,Z, \Enabla Z)
- (\Div H)_k -  H'_k \right],$$
in  the set $E_r$ defined in \eqref{eqdefEr},
with respect to the norm defined in \eqref{eqdefEnorm}. For any $Z\in E_r$ we have
\begin{equation}\label{eqZLinftyE}
 \|Z\|_{L^\infty}\le C r^{1-\frac nm}\Enorm{Z}\le  C r^{1-\frac nm},
\end{equation}
therefore if $Cr^{1-\frac nm}< r_0\Lambda$ the
function $\Xi(x, Z, \Enabla Z)$ is well defined for $Z\in E_r$.

The operator $(-L_g)^{-1} $ is a linear, bounded map from $L^s(B_r)$ to $W^{2,s}\cap W^{1,s}_0(B_r)$
which extends to a bounded map from $W^{-1,s}(B_r)
:=(W^{1,s'}_0(B_r))'$ (with the homogeneous norm) to $W^{1,s}_0(B_r)$,
for any $s\in (1,\infty)$. The map does not depend on $s$, the bounds depend on $s$ but not on $r$.
Specifically, we have, for some $C(s)$ which depends only on $s$, $n$, and $C_0$,
\begin{equation}\label{eqellipticLgWmus}
 \|\Enabla (-L_g)^{-1}u\|_{L^s(B_r)} \le C(s)
 \|u\|_{W^{-1,s}(B_r)}
\end{equation}
for all $u\in W^{-1,s}(B_r)$ (see Lemma~\ref{lemmanablaLgW1ms})
and
\begin{equation}\label{eqestimateLgw2s}
 \|\Enabla^2 (-L_g)^{-1}u\|_{L^s(B_r)} \le C(s)
 \|u\|_{L^{s}(B_r)}
\end{equation}
for all $u\in L^s(B_r)$ (see Lemma~\ref{lemmanablaLgW2s}).
We define $\alpha\in(1,m)$  by
$$ \frac1\alpha := \frac1m + \frac1n;
$$
for some
$\beta=\beta(\alpha,m)>0$
by H\"older's inequality
\begin{equation}  \label{eq:hoelder_q_half_m}
\| u\|_{L^\alpha} \le C r^{n \beta} \|u \|_{L^{m/2}}.
\end{equation}
From \eqref{eqestimateLgw2s} with $s=m$  we obtain
\begin{equation}  \label{eq:Lq_estimate2}
\Enorm{\Enabla (-L_g)^{-1} u}
=
\|\Enabla^2 (-L_g)^{-1} u\|_{L^m} \le C  \|u \|_{L^m};
\end{equation}
similarly using
\eqref{eqestimateLgw2s} with $s=\alpha$ and
the critical Sobolev embedding $W^{1,\alpha} \hookrightarrow L^m$,
\begin{equation}  \label{eq:Lq_estimate}
\Enorm{(-L_g)^{-1} u} =
\|\Enabla (-L_g)^{-1} u\|_{L^m}
\le
C \|\Enabla^2 (-L_g)^{-1} u\|_{L^\alpha}
\le C \|u \|_{L^\alpha}.
\end{equation}
From \eqref{eqestimateLgw2s} with $s=m$
and Poincar\'e's inequality, which we can apply since $\Enabla (-L_g)^{-1} H'$ has average zero, we obtain
\begin{equation}\label{eqLgHlm}
\|\Enabla (-L_g)^{-1} H'\|_{L^m}
\le
Cr \|\Enabla^2 (-L_g)^{-1} H'\|_{L^m}
\le Cr \|H' \|_{L^m}.
\end{equation}
Since $T(0)=(-L_g)^{-1}(-\Div H-H')$,
using \eqref{eqLgHlm} and
\eqref{eqellipticLgWmus} with $u=-\Div H$ gives
\begin{equation}\label{eqenormt0lm}
\Enorm{T(0)}
 \le C  \|H\|_{L^m(B_r)}
 + C r \|H'\|_{L^m(B_r)}.
\end{equation}
Recalling \eqref{eqassHHstrich},
we obtain
$\Enorm{ T(0) } \le
C \Lambda r_1^{{n}/{m}}$, therefore if $r_1$ is sufficiently small
\begin{equation}  \label{eq:estimate_T_zero}
\Enorm{ T(0) } \le
\frac12.
\end{equation}
To show the existence of a fixed point in $E_r$
it thus suffices to
show that for $r_1\in(0, 1\wedge r_0]$ sufficiently small one can obtain, for any $r\in (0, r_1]$, $Z$, $Z' \in E_r$ the bound
\begin{equation}  \label{eq:contraction_boundary_fixed}
\Enorm{ T(Z) - T(Z') } \le \frac12  \Enorm{Z - Z'}.
\end{equation}
Using \eqref{eqHaAbBlip},
\begin{equation*}\begin{split}
&\hskip-1cm\|  \Xi(\cdot, Z(\cdot), \Enabla Z(\cdot)) -   \Xi(\cdot, Z'(\cdot), \Enabla Z'(\cdot)) \|_{L^{m/2}(B_r)}\\
\le&  C \|Z-Z'\|_{L^\infty(B_r)} (r^{2n/m} + \|\Enabla Z\|_{L^m(B_r)}^2+\|\Enabla Z'\|_{L^m(B_r)}^2)\\
&+ C  \|\Enabla Z-\Enabla Z'\|_{L^m(B_r)}
(r^{n/m}+ \|\Enabla Z\|_{L^m(B_r)}+\|\Enabla Z'\|_{L^m(B_r)}).
\end{split}
\end{equation*}
{Recalling that $r\le r_1\le r_0$,}
and using
the embedding $W_0^{1,m} \hookrightarrow L^\infty$ we get
$$ \|  \Xi(\cdot, Z(\cdot), \Enabla Z(\cdot)) -   \Xi(\cdot, Z'(\cdot), \Enabla Z'(\cdot)) \|_{L^{m/2}(B_r)}
\le C   \Enorm{Z - Z'}.$$
Hence with   \eqref{eq:hoelder_q_half_m} and \eqref{eq:Lq_estimate} we get
 $$   \Enorm{ T(Z) - T(Z') } \le
 C r_1^{n\beta} \Enorm{Z - Z'}.$$
Thus  \eqref{eq:contraction_boundary_fixed} follows for $r_1 > 0$ sufficiently small.

It follows from  \eqref{eq:estimate_T_zero} and  \eqref{eq:contraction_boundary_fixed}
that
for any $r\in(0,r_1]$ the map
$T$ has a unique fixed point $Z_*$ in $E_r$, which satisfies
$\|Z_*\|_{E_r}=\|T(Z_*)\|_{E_r}\le \|T(0)\|_{E_r}+\|T(Z_*)-T(0)\|_{E_r}$ and therefore
with~\eqref{eqenormt0lm} and~\eqref{eqassHHstrich}
\begin{equation}  \label{eq:W1m_bound_fixed_boundary}
\|  \Enabla Z_* \|_{L^m(B_r)}  \le 2 \|  T(0)\|_{E_r}
 \le C  \|H\|_{L^m(B_r)}
 + C r\|H'\|_{L^m(B_r)} \le C \Lambda r^{\frac{n}{m}}.
\end{equation}

\medskip

\emph{Step 2: Bound for the $W^{1,\ell}$ norm.}\\
We prove the bound \eqref{eqLpboundnablaZh} for the $W^{1,\ell}$ norm of $Z_*$ in terms of the $L^\ell$ norm of $H$ and $H'$. We denote here $Z_*$ simply by $Z$, all norm in this step are taken on $B_r$.
Analogously to~\eqref{eq:Lq_estimate2}-\eqref{eqLgHlm},
from Poincar\'e's inequality
and \eqref{eqestimateLgw2s} with $s=\ell$
we obtain
\begin{equation}
\|\Enabla (-L_g)^{-1} H'\|_{L^\ell}
\le
Cr \|\Enabla^2 (-L_g)^{-1} H'\|_{L^\ell}
\le Cr \|H' \|_{L^\ell}.
\end{equation}
Similarly, using \eqref{eqellipticLgWmus} with $s=\ell$ gives
\begin{equation}
 \|\Enabla (-L_g)^{-1}\Div H\|_{L^\ell}
\le C \|\Div H\|_{W^{-1,\ell}}
 \le C \|H\|_{L^\ell}.
\end{equation}
Therefore
$$ \| \Enabla T(0)\|_{L^\ell} \le C
  \|H\|_{L^\ell}
 + Cr \|H'\|_{L^\ell}
.$$
Since $Z=T(Z)$,
to obtain  \eqref{eqLpboundnablaZh}
it suffices to show that
\begin{equation} \label{eq:contraction_p}
 \| \Enabla T(Z) - \Enabla T(0)\|_{L^\ell} \le \frac12 \| \Enabla Z\|_{L^\ell} \quad \text{for all $Z \in E_r$.}
 \end{equation}
 Using  \eqref{eqHaAbBlip} with $b=0$, $B=0$ and the fact that $ \Xi(x,0,0) = 0$ we get
 \begin{equation}\label{eqestHZDZ}
| \Xi(x,Z(x), \Enabla Z(x))| \le C|Z|(x)(1+|\Enabla Z(x)|^2)+C|\Enabla Z(x)|(1 + |\Enabla Z(x)|).
\end{equation}
For any $s\in [1,\infty)$,
using ${\|Z\|_{L^\infty}}\le Cr^{1-n/m}\Enorm{Z}$,
 Poincar\'e's inequality in $W^{1,s}_0$, and $r\le1$,
for $Z \in E_r$ we obtain
\begin{equation*}\begin{split}
\| \Xi(\cdot,Z(\cdot), \Enabla Z(\cdot))\|_{L^s} \le&
C\|Z\|_{L^s}
+ C \|\Enabla Z\|_{L^{s}}
+ C \|\Enabla Z\|_{L^{2s}}^2
\\
\le & C \|\Enabla Z\|_{L^{s}}
+ C \|\Enabla Z\|_{L^{2s}}^2.
\end{split}\end{equation*}
The key difficulty is treating the quadratic term; on the positive side we have a uniform bound in $L^m$ and the fact that the operator gains one derivative.
We  choose the exponent  $s$ such that
\begin{equation}\label{eqchoices}
1<s<n,
\hskip5mm
s<\ell<s^*:=\frac{ns}{n-s},
\hskip5mm
\frac1s > \frac1\ell+\frac1m.
\end{equation}
This is possible since $1<n<m$,
$p\le \ell$, and
$\frac1p+\frac1m<1$,
imply that
one can choose $s$ with
\begin{equation*}
 \frac1n \vee \left(\frac1\ell+\frac1m\right)<\frac1s<1\wedge \left(\frac1\ell+\frac1n\right),
\end{equation*}
which is equivalent to \eqref{eqchoices} (the inequality $s<\ell$ follows from the last one in \eqref{eqchoices}).

Using \eqref{eqchoices},
Hölder's inequality, and $r\le 1$ we obtain
\begin{equation}
\|\Enabla Z\|_{L^{s}}\le Cr^\gamma \|\Enabla Z\|_{L^{\ell}}
\text{ and }\|\Enabla Z\|_{L^{2s}}^2\le
Cr^{\gamma} \|\Enabla Z\|_{L^{\ell}}\|\Enabla Z\|_{L^{m}}
\end{equation}
for some $\gamma>0$ that depends only on $s$, $n$, $\ell$, $m$. Therefore for any $Z\in E_r$
\begin{equation*}\begin{split}
\| \Xi(\cdot,Z(\cdot), \Enabla Z(\cdot))\|_{L^s} \le & C r^{\gamma} \|\Enabla Z\|_{L^{\ell}}.
\end{split}\end{equation*}
At this point we use \eqref{eqestimateLgw2s} and the embedding $L^\ell\subset
L^{s^*}\subset W^{1,s}$
and obtain
\begin{equation*}\begin{split}
  \| \Enabla T(Z) - \Enabla T(0)\|_{L^\ell} \le& Cr^{\gamma'}
 \| \Enabla^2(-L_g)^{-1} \Xi(\cdot,Z(\cdot), \Enabla Z(\cdot))
  \|_{L^s} \\
  \le&Cr^{\gamma'}
  \|  \Xi(\cdot,Z(\cdot), \Enabla Z(\cdot))
  \|_{L^s} \le
  C r_1^{\gamma
+\gamma'}  \| \Enabla Z\|_{L^\ell}
\end{split}\end{equation*}
for some $\gamma'>0$.
Therefore for sufficiently small $r_1>0$ the condition \eqref{eq:contraction_p} holds, and \eqref{eqLpboundnablaZh} follows.

\emph{Step 3: Bound for the $L^{\infty}$ norm.}\\
We prove
\eqref{eqlpboundlinfty} using
Lemma~\ref{le:C1_C2_estimate}.
We set $C'=1$, $u=Z$, $a=\sqrt{\det g} g$, $G=\sqrt{\det g}\Xi$, $d=n$.
By~\eqref{eqgC0}, $a$ satisfies
\eqref{eq:app_ellipticity_a}-\eqref{eq:app_C2_a}  with constants that depend only on $C_0$.
Since $r_1\le 1$, $n/m<1$, and
$Z\in E_r$,  \eqref{eq:app_apriori_bound_Du_Lm} follows.
By~\eqref{eqassHHstrich} and~\eqref{eqLpboundnablaZh} with $\ell=m$, we obtain
\begin{equation}
 \|DZ\|_{L^m(B_r)}\le C_1 \Lambda r^{n/m}.
\end{equation}
By Lemma~\ref{le:C1_C2_estimate} used on the ball $B_r$,  we obtain
that \begin{equation}\label{eqlinftybound}
      \|DZ\|_{L^\infty(B_{3r/4})}
      \le C_2\Lambda,
     \end{equation}
with $C_2$ depending only on the parameters.
This proves~\eqref{eqlpboundlinfty}.
\end{proof}

In closing, we show that the differential operator $-L_g$ is invertible, with uniform estimates of the form given in
\eqref{eqellipticLgWmus}-\eqref{eqestimateLgw2s}. For the use above it is important that the bounds depend only on the exponent, the dimension, and the constant $C_0$ entering \eqref{eqgC0}, but not on the function $g$, so that the dependence on $r$ becomes explicit.
We intend to solve an equation of the form
\begin{equation}\label{eqLgufinalpart1}
-L_g u(x) =
-\partial_i  H_{ki}(x)- H'_k(x),
\end{equation}
where $L_g$ was defined in \eqref{eqdefLgfromg}, which means that
\begin{equation}
- \frac1{\sqrt {\det g(x)}}\partial_i (\sqrt{\det g(x)} g^{ij}(x) \partial_j  u(x))=-\partial_i  H_{ki}(x)- H'_k(x).
\end{equation}
Rearranging terms this is the same as
\begin{equation}\label{eqLgufinalpart3}
\begin{split}
- \partial_i (\sqrt{\det g(x)} g^{ij}(x) \partial_j  u(x))=&-\partial_i  ({\sqrt{\det g(x)}}H_{ki}(x))
+H_{ki}(x) \partial_i  {\sqrt{\det g(x)}}\\
&- {\sqrt{\det g(x)}} H'_k(x).
\end{split}
\end{equation}

\begin{lemma}\label{lemmanablaLgW1ms}
 For $C_0>0$, $s\in (1,\infty)$, there is $C>0$ such that
 for any $r\in(0,1]$, any $g\in C^2(B_r)$ which obeys~\eqref{eqgC0} in $B_r$, one has 
\begin{equation}\label{eqellipticLgWmus1}
 \|\Enabla (-L_g)^{-1}u\|_{L^s(B_r)} \le C(s)
 \|u\|_{W^{-1,s}(B_r)}
\end{equation}
for all $u\in W^{-1,s}(B_r)$.
\end{lemma}
\begin{proof}
Assume the Lemma holds for $r=1$, and let $C$ be the resulting constant. Let $g\in C^2( B_r)$
obey~\eqref{eqgC0}. We define $\hat g\in C^2(B_1)$
by $\hat g(x)=g(rx)$. Then $|D\hat g|\le r |Dg|\le C_0$, and the same for $D^2\hat g$. Therefore $\hat g$ obeys~\eqref{eqgC0} with the same $C_0$. We set $\hat u(x)=u(rx)$, let $\hat U:=(-L_{\hat g})^{-1}\hat u$,
so that $\|\Enabla \hat U\|_{L^s(B_1)}\le C(s)\|\hat u\|_{W^{-1,s}(B_1)}$,
and set $U(x):=r\hat U(x/r)$. We check that
$-L_g U=u$, and that the estimate follows.

Therefore it suffices to prove the Lemma for $r=1$.
We recall that for any $u\in W^{1,-s}$ there is
$v\in L^s$ with $\Div v=u$ and $\|v\|_{L^s}\le C \|u\|_{W^{-1,s}}$.
We write
$-L_g U = \Div v$ as
\begin{equation}
- \Div (\sqrt{\det g} g^{-1} DU)= \Div (v {\sqrt{\det g}}) - v\cdot \Enabla \sqrt{\det g},
\end{equation}
which has the same form as \eqref{eq:app_scalar_divergence_form}. The coefficients and their derivatives (and, hence, their modulus of continuity) are controlled by $C_0$. Then from Lemma~\ref{le:w1p_dirichlet}  with $p=s$ we obtain
\begin{equation}
 \|\Enabla U\|_{L^s(B_1)} \le C(s)
 \|v\|_{L^s(B_1)}.
\end{equation}
Recalling $\|v\|_{L^s(B_1)}\le C(s)
 \|u\|_{W^{1,-s}(B_1)}$,
observing that $U=(-L_v)^{-1}u$, this has the same form as \eqref{eqellipticLgWmus}.
\end{proof}

\begin{lemma}\label{lemmanablaLgW2s}
 For $C_0>0$, $s\in (1,\infty)$, there is $C>0$ such that
 for any $r\in(0,1]$, any $g\in C^2( B_r)$ which obeys~\eqref{eqgC0} in $B_r$  one has
\begin{equation}\label{eqellipticLgW2s}
 \|\Enabla^2 (-L_g)^{-1}u\|_{L^s(B_r)} \le C(s)
 \|u\|_{L^s(B_r)}
\end{equation}
for all $u\in L^{s}(B_r)$.
\end{lemma}
\begin{proof}As in the other case, by scaling it suffices to prove the Lemma for $r=1$, by density we can assume $u\in C^\infty_c$.
We write
$-L_g U = u$ as
\begin{equation}
- \Div (\sqrt{\det g} g^{-1} DU)= u {\sqrt{\det g}},
\end{equation}
we know that it has a unique solution $U$, with
\begin{equation}\label{eqnablaUCu}
 \|\Enabla U\|_{L^s(B_1)}\le C \|u\|_{W^{-1,s}(B_1)}
 \le C \|u\|_{L^s(B_1)}.
\end{equation}
We differentiate the equation and obtain
\begin{equation}
- \Div D_j  (\sqrt{\det g} g^{-1} DU)= D_j (u {\sqrt{\det g}}),
\end{equation}
which can be rewritten as
\begin{equation}
- \Div (\sqrt{\det g} g^{-1} D(D_j  U))= D_j (u {\sqrt{\det g}})
- \Div (D_j(\sqrt{\det g} g^{-1}) D U)
\end{equation}
By the previous Lemma,
\begin{equation}\begin{split}
 \|\Enabla D_jU\|_{L^s} \le& C \|u\sqrt{\det g}\|_{L^{s}}
 + C\|D_j(\sqrt{\det g} g^{-1}) D U\|_{L^s}\\
  \le & C \|u\|_{L^{s}}
 + C\| D U\|_{L^s}\end{split}
\end{equation}
Recalling~\eqref{eqnablaUCu}, the proof is concluded.
\end{proof}

\section{Weighted estimate at the boundary}\label{se:smith}
In this section we prove bounds for functions that satisfy estimates of the form
\begin{equation}
|D^2f| \le \frac{A}{\dist(\cdot,\partial\Omega)} + B .
\end{equation}
It would be tempting to simply solve this by
taking a primitive separately of the two components of $D^2f$. However, they are not separately second gradients, only their sum is. While one may think of this as a decomposition into low- and high-frequencies, the precise construction of the decomposition is based on a singular-integral formulation.

Precisely, we prove the following.

\begin{proposition}\label{propLipschitzboundaryfull}
Let $\Omega \subset \R^n$ be a bounded Lipschitz set,
$1<p<m<\infty$,
$f\in W^{1,m}(\Omega)\cap W^{2,1}_\mathrm{loc}(\Omega)$, 
$A,B:\Omega\to[0,\infty)$ be such that
\begin{equation}\label{eqD2fAdistB}
|D^2f| \le \frac{A}{\dist(\cdot,\partial\Omega)} + B ,
\end{equation}
and, for some $\alpha\ge0$, $\eps\ge0$,
\begin{equation}
 \|f\|_{W^{1,m}}
 +\|A\|_{L^m}
 +\|B\|_{L^m}\le \alpha,\quad
 \|A\|_{L^p}\le \eps.
\end{equation}
Then there are $f_h$, $f_l\in W^{1,m}(\Omega)$ such that $f=f_h+f_l$,
\begin{equation}
 \|f_h\|_{W^{1,p}}\le C \eps,\quad
 \|f_h\|_{W^{1,m}}+
 \|f_l\|_{W^{2,m}}\le C\alpha
\end{equation}
with $C=C(\Omega,p,m)$.
\end{proposition}
Actually, we will only need the following local version, which implies Proposition~\ref{propLipschitzboundaryfull} by a standard covering argument with an appropriate partition of unity.
\newcommand\omegarf{\omega_{r,\gamma}}%
\newcommand\omegarfi{\omega_{r_i,\gamma_i}}%
\newcommand\omegarhaf{\omega_{r/2,\gamma}}%
\newcommand\omegarhafi{\omega_{r_i/2,\gamma_i}}%
\newcommand\omegarduefi{\omega_{2r_i,\gamma_i}}%
\newcommand\omegarfj{\omega_{r_j,\gamma_j}}%
\newcommand\omegarhafj{\omega_{r_j/2,\gamma_j}}%
\newcommand\omegarduefj{\omega_{2r_j,\gamma_j}}%
For $r,L>0$ we define
\begin{equation}\label{eqdefCrL}
T_{r,L}:= B_r'\times (-2Lr,2Lr),
\end{equation}
as usual $B_r'$ denotes the ball in $\R^{n-1}$.
\begin{proposition}\label{propLipschitzboundarylocal}
Let
$1<p<m<\infty$,
$r>0$, $L>0$, $\gamma\in\Lip(\R^{n-1})$ with $\Lip(\gamma)\le L$ and $\gamma(0)=0$,
\begin{equation}\label{eqdefomegargamma}
 \omegarf:=\{(x',x_n)\in T_{r,L}: x_n<\gamma(x')\},
\end{equation}
$f\in W^{1,m}(\omegarf)\cap W^{2,1}_\mathrm{loc}(\omegarf)$, with
\begin{equation}\label{eqD2fAdistBloc}
|D^2f|(x) \le \frac{A(x)}{\gamma(x')-x_n} + B(x) ,
\end{equation}
\begin{equation}
 \|f\|_{W^{1,m}(\omegarf)}
 +\|A\|_{L^m(\omegarf)}
 +\|B\|_{L^m(\omegarf)}\le \alpha,\quad
 \|A\|_{L^p(\omegarf)}\le \eps.
\end{equation}
Then there are $f_h$, $f_l\in W^{1,m}(\omegarhaf)$ such that $f=f_h+f_l$,
\begin{equation}
 \|f_h\|_{W^{1,p}(\omegarhaf)}\le C \eps,\quad
 \|f_h\|_{W^{1,m}(\omegarhaf)}+
 \|f_l\|_{W^{2,m}(\omegarhaf)}\le C\alpha
\end{equation}
with $C=C(p,m,r,n,\gamma)$.
\end{proposition}

\begin{proof}[Proof of Prop.~\ref{propLipschitzboundaryfull} from Prop.~\ref{propLipschitzboundarylocal}]

Since $\Omega$ is bounded and Lipschitz, there are finitely many functions $\gamma_j\in \Lip(\R^{n-1})$, $\gamma_j(0)=0$,
$r_j>0$,
and affine isometries $I_j:\R^n\to\R^n$ such that, using the notation from \eqref{eqdefomegargamma},
\begin{equation}\label{eqLipomegaIi}
\Omega\cap  I_j(T_{2r_j,L})=I_j(\omegarduefj),
\hskip5mm
\partial\Omega\subset \bigcup_{j=1}^J
I_j(T_{r_j/2,L}).
\end{equation}
Let $\theta_j$, $j=0,\dots, J$, be a partition of unity on $\Omega$ with $\theta_0\in C^\infty_c(\Omega)$,
$\theta_j\in C^\infty_c( I_j(T_{r_j,L}))$,
and set $f_j:=f\theta_j$. We shall decompose each of them, and then add the results.

For $j=0$,
from $\dist(\supp\theta_0,\partial\Omega)>0$ we obtain
 $\|f_0\|_{2,m}\le C \|f\|_{1,m}+ \|\theta_0 D^2f\|_m\le C\alpha$, so that $f_{l,0}:=f_0$, $f_{h,0}:=0$ will do.

We next
treat one of the boundary terms $f_j=f\theta_j$.
For simplicity of notation assume $I_j$ is the identity.
 We define $\hat A_j:\omegarfj\to\R$ by
 \begin{equation}
  \hat A_j(x',x_n):=(\theta_jA)(x',x_n)  \frac{\gamma_j(x')-x_n}{\dist(x,\partial\Omega)},
 \end{equation}
and observe that $|\hat A_j|\le C \theta_j|A|$ pointwise
in $\omega_{r,\gamma_j}$.
To obtain this bound, one uses that the first condition in~\eqref{eqLipomegaIi}
involves $\omega_{2r,\gamma_j}$.
At the same time,
\begin{equation}\begin{split}
 |D^2f_j|\le & \theta_j |D^2f|+ 2|D\theta_j|\, |Df|+|D^2\theta_j|\, |f|\\
 \le& \frac{\theta_j A}{\dist(\cdot,\partial\Omega)}
 +\theta_j B + 2|D\theta_j|\, |Df|+|D^2\theta_j|\, |f|,
\end{split}\end{equation}
so that setting
$\hat B_j:=\theta_j B + 2|D\theta_j|\, |Df|+|D^2\theta_j|\, |f|$
we have the desired bound
\begin{equation}\begin{split}
 |D^2f_j|(x)\le
& \frac{\hat A_j(x)}{\gamma_j(x')-x_n}
 +\hat B_j(x).
\end{split}\end{equation}
One immediately checks that all estimates are inherited (up to constants).
From \Cref{propLipschitzboundarylocal}
we obtain $f_{j,h}$ and $f_{j,l}$.
Before summing, we need to extend those functions to the rest of $\Omega$.
For this we select for each $j$ a cutoff function
$\psi_j\in  C^\infty_c( I_j(T_{r_j,L}))$,
with $\psi_j=1$ on $\supp \theta_j$. Since $f_j=0$
in $\Omega\setminus \supp\theta_j$, we have
$f_j=\psi_j f_j = \psi_j f_{j,h}+ \psi_j f_{j,l}$ pointwise in $\Omega$ (the function $\psi_j f_{j,h}$ is implicitly extended by 0 to the rest of $\Omega$).
For $j=0$ we set $\psi_0=1$.
Finally, we set
\begin{equation}
 f_h:=\sum_{j=0}^J \psi_j f_{j,h},\hskip1cm
 f_l:=\sum_{j=0}^J \psi_j f_{j,l}
\end{equation}
and easily check that they have the desired properties.
\end{proof}

To prove Proposition~\ref{propLipschitzboundarylocal}, we recall a general representation formula to recover $u$ from its second derivatives (Proposition~\ref{propsmith}), then we show an estimate for solutions of equations of the form
$D^2u=\Div A+B$
(Proposition~\ref{propLipschitzboundary}), and finally show that the weighted estimate \eqref{eqD2fAdistB} can be converted to a bound of this form (Lemma~\ref{lemmaxntodiv}).

We use the following  classical result to reconstruct a function from its second partial derivatives.
The result holds much more generally for families of linear  homogeneous differential operators with constant coefficients whose characteristic polynomials have no common complex zeroes,
see  \cite[Th. I]{Smith1970} and   \cite{Smith1961}.

\begin{proposition}\label{propsmith}
Let $\Lambda \subset \R^n$ be  a closed cone with nonempty interior. Then for $i, j \in \{1, \ldots, n\}$ there exist
functions $K_{ij} : \R^n \setminus \{0\} \to \R$ with the following properties:
\begin{enumerate}
\item $K_{ij} \in C^\infty(\R^n \setminus \{0\}$) and
$K_{ij}$ is homogeneous of degree $2-n$;
\item $K_{ij} =0$ on $\R^n \setminus \Lambda$;
\item for every $u \in C_c^\infty(\R^n)$ we have 
$u = \sum_{i,j=1}^n K_{ij} \ast \partial_i \partial_j u$, i.e.,
\begin{equation}  \label{eq:representation_u}
u(x) = \sum_{i,j=1}^n\int_{\R^n}  K_{ij}(y) \, \partial_i \partial_j u(x-y) \, dy.
\end{equation}
\end{enumerate}
\end{proposition}

\begin{proof} We include the short proof for the convenience of the reader.
Since $\Lambda$ has nonempty interior,
the intersection $\Lambda \cap S^{n-1}$ contains a (relatively) open set $V \subset S^{n-1}$. 
We will reconstruct $u$ from its second directional derivatives for directions $\xi \in V$ and average over $V$, $K$ will depend only on $V$.
It suffices to show  \eqref{eq:representation_u} for $x = 0$.

For any $w \in C_c^\infty(\R)$ we have
$$ w(0) = \int_0^\infty r w''(r) \, dr.$$
Applying this to $r \mapsto u(r\xi)$ with $\xi \in V$ we get
\begin{equation}   \label{eq:smith_1d}  u(0) = \sum_{ij}\int_0^\infty  r \xi_i \xi_j  (\partial_i \partial_j  u)(r \xi) \, dr.
\end{equation}
Let $\eta \in C_c^\infty(V)$ with $\int_{S^{n-1}} \eta \, d\mathcal H^{n-1} = 1$.
Multiplying  \eqref{eq:smith_1d}  by $\eta(\xi)$, integrating over $\xi$ and switching from polar coordinates
to Euclidean coordinates we get
\begin{align*}
u(0) = & \,\sum_{ij}  \int_0^\infty \int_{S^{n-1}}  r \eta(\xi) \xi_i \xi_j (\partial_i  \partial_j u)(r \xi) d\mathcal H^{n-1} \, dr \\
= & \, \sum_{ij}\int_{\R^n}  |x|^{2-n} \eta\left(\frac{x}{|x|}\right) \frac{x_i}{|x|} \frac{x_j}{|x|} \, \partial_i \partial_j u(x) \, dx.
\end{align*}
Thus   \eqref{eq:representation_u} holds with $K_{ij}(x) := |x|^{-n} x_i x_j \eta(x/|x|)$.
\end{proof}

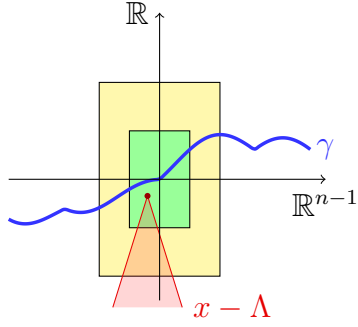
\begin{figure}[h]
\def\rr{.4}
\def\LL{1.6}
\begin{center}
 \begin{tikzpicture}[scale=2]
  \draw [fill=yellow!40!white]
  (-\rr,{-\LL*\rr}) rectangle (\rr,{\LL*\rr});

  \draw [fill=green!40!white]
  (-0.5*\rr,{-0.5*\LL*\rr}) rectangle (0.5*\rr,{0.5*\LL*\rr});

    \draw[line width=1.3,blue!80!white,domain=-1:1,samples=100,smooth]
    plot (\x,{0.25*sin(2.2*\x r) + sin((.6+2*\x ) r)*0.15*abs(sin(5*\x r)) + -.08*\x})
    ++(.1,0) node {$\gamma$};

  \def\cone{.23}
  \draw (-.08,-.11) coordinate (A);
  \draw[red!90!black] (A) -- ++ (\cone,-\cone*2*\LL)
  node[right] {$x-\Lambda$}
  (A) -- ++ (-\cone,-\cone*2*\LL) ;

  \fill[red!80!white,opacity=0.2]
  (A)--  ++ (\cone,-\cone*2*\LL)  -- ++ (-2*\cone,0) -- cycle;

  \fill[red!60!black] (A) circle (.02);

  \draw[->](-1,0) -- (1.1,0) node [below] {$\R^{n-1}$};
  \draw[->](0,-.8) -- (0,1.1) node [left] {$\R$};
 \end{tikzpicture}
 \end{center}
\caption{Sketch of the geometry in the proof of Proposition~\ref{propLipschitzboundary}. The two rectangles are $T_r$ and $T_{r/2}$, the red dot is $x\in \omegarhaf$.}
\label{figlipschitz}
\end{figure}

\begin{proposition}\label{propLipschitzboundary}
Let $1<p<m<\infty$, $r>0$, $L>0$, $\gamma\in \Lip(\R^{n-1})$ with $\Lip(\gamma)\le L$ and $\gamma(0)=0$, let $\omegarf$ be as in \eqref{eqdefomegargamma}.
Let $f\in W^{1,m}(\omega_{r,\gamma})\cap W^{2,1}_\mathrm{loc}(\omega_{r,\gamma})$, with
\begin{equation}
 D^2f=\Div A + B
\end{equation}
distributionally, for some $A\in L^m(\omegarf;\R^{n\times n\times n})$, $B\in L^m(\omegarf;\R^{n\times n})$ and assume
\begin{equation}
 \|f\|_{W^{1,m}(\omega_{r,\gamma})}
 +\|A\|_{L^m(\omega_{r,\gamma})}
 +\|B\|_{L^m(\omega_{r,\gamma})}\le \alpha,\quad
 \|A\|_{L^p(\omega_{r,\gamma})}\le \eps.
\end{equation}
Then there are $f_h$, $f_l\in W^{1,m}(\omegarhaf)$ such that $f=f_h+f_l$,
\begin{equation}
 \|f_h\|_{W^{1,p}(\omegarhaf)}\le C \eps,\quad
 \|f_h\|_{W^{1,m}(\omegarhaf)}+
 \|f_l\|_{W^{2,m}(\omegarhaf)}\le C\alpha,
\end{equation}
with $C=C(p,m,r,n,L)$.
\end{proposition}
\begin{proof}
We recall \eqref{eqdefCrL},
and write briefly $T_r:=T_{r,L}= B_r'\times (-2Lr,2Lr)$ since $L$ is fixed,
so that $\omegarf=T_r\cap \{ y_n< \gamma(y')\}$, see Figure~\ref{figlipschitz}.
 Let $K$ be the kernel
from Proposition~\ref{propsmith} for the cone
 \begin{equation}
  \Lambda:=\{(y',y_n): L|y'|\le y_n\}.
 \end{equation}
 Fix $\eta\in C^\infty_c(T_r)$ such that $\eta=1$ on $T_{3r/4}$.
We observe that
\begin{equation}
 D^2(\eta f)=\eta (\Div A+B) + D\eta\otimes Df+Df\otimes D\eta + fD^2\eta.
\end{equation}
We set
\newcommand\Beta{B^{(\eta)}}
\begin{equation}
 B':=\eta B  + D\eta\otimes Df+Df\otimes D\eta + fD^2\eta- A\cdot\Enabla \eta ,
\end{equation}
so that $D^2(\eta f)=\Div(\eta A)+B'$
and
\begin{equation}
 \|B'\|_{L^m(\omegarf)}\le \|B\|_{L^m(\omegarf)}+C \|f\|_{W^{1,m}(\omegarf)}
 +C\|A\|_{L^m(\omegarf)}
 \le C \alpha.
\end{equation}
The constants may depends on $r$.
We check that
\begin{equation}\label{eqxyinclusion}
\text{if $x\in\omegarhaf$, $y\in \Lambda$, and $x-y\in T_r$ , then $x-y\in \omegarf$.}
\end{equation}
Indeed, $\Lip(\gamma)\le L$ and the definitions imply
 \begin{equation}
  (x-y)_n<\gamma(x')-y_n\le \gamma(x')-L|y'|\le \gamma((x-y)').
 \end{equation}
By Proposition~\ref{propsmith}, if $f\in C_c^\infty(\R^n)$ then for $x\in\omegarhaf$
 we have
\begin{equation}\label{eqfetaf}
 f(x)=(\eta f)(x)=\sum_{ij}\int_\Lambda K_{ij}(y) \partial_i\partial_j (\eta f)(x-y) dy,
\end{equation}
and in the right-hand side we can insert $\Div(\eta A)+B'$ in place of $D^2(\eta f)$, obtaining the desired representation.
Therefore we define
\begin{equation}\label{eqdeffhfl}
f_h(x):=\sum_{ijk}\partial_i \int_\Lambda K_{jk}(y) (\eta A_{jki})(x-y) dy,
\hskip5mm
f_l:=K\ast B',
\end{equation}
which, if regularity allows to
swap derivation and integration, is the same as $f_h=K\ast \Div (\eta A)=(DK)\ast (\eta A)$.

In the general case, $f_h$ and $f_l$ are still defined by \eqref{eqdeffhfl}, but proving $f=f_h+f_l$ requires approximation.
For $\eps>0$
we fix a mollifier $\varphi_\eps\in C^\infty_c(B_\eps)$
and define
$\hat f_\eps:=\varphi_\eps\ast (\eta f \chi_{\omegarf})$, $\hat B_\eps:=\varphi_\eps\ast (B'\chi_{\omegarf})$,
$\hat A_\eps:=\varphi_\eps\ast (\eta A\chi_{\omegarf})$.
Choose $\delta\in(0,\dist(\supp\eta, \partial T_r)\wedge r/4)$, and assume $\eps\in(0, \delta/(L+1))$.
 As in \eqref{eqfetaf}, Proposition~\ref{propsmith}
for $\hat f_\eps\in C^\infty_c(T_r)$
 implies
\begin{equation}\label{eqpropfabeps}
\hat f_\eps(x)=
\sum_{jk}\int_\Lambda K_{jk}(y)(D^2\hat f_\eps)_{jk}(x-y) dy.
\end{equation}
Assume now that $x\in \omega_{r/2,\gamma-\delta}$.
By the choice of $\delta$ and $\eps$,
$\hat f_\eps(x)=(\varphi_\eps\ast f)(x)$. By
\eqref{eqxyinclusion} applied to $\gamma-\delta$ (its proof did not use $\gamma(0)=0$),
we can restrict the integral to the set for which $x-y\in \omega_{r/2,\gamma-\delta}$.
Again by the choice of $\delta$ and $\eps$,
$D^2\hat f_\eps=\Div \hat A_\eps+\hat B_\eps
=\Div (\varphi_\eps\ast A)+\varphi_\eps\ast B$
in $\omega_{r/2,\gamma-\delta}$. Inserting in
\eqref{eqpropfabeps} leads to
\begin{equation}\label{eqpropfab}
\varphi\ast f=
\varphi_\eps \ast f_h +\varphi_\eps\ast f_l
\end{equation}
in $\omega_{r/2,\gamma-\delta}$. Letting $\eps\to0$ we obtain
$f=f_h+f_l$ pointwise almost everywhere in
$\omega_{r/2,\gamma-\delta}$, by arbitrariness of $\delta$ it holds almost everywhere
in $\omega_{r/2,\gamma}$.

It remains to obtain the $L^p$ and $L^m$ estimates.
As $K$ is homogeneous of degree $2-n$,
its second derivatives are $-n$-homogeneous kernels with average zero, and with derivatives which are $-n-1$-homogeneous. Therefore standard estimates for singular convolution kernels
(see, for example, \cite[Theorem~2 on page~35]{Stein1970}) give
\begin{equation*}
 \|\Enabla f_h\|_{L^{p}(\omegarhaf)} \le C \eps,  \quad  \| \Enabla f_h\|_{L^{m}(\omegarhaf)}   +
 \| \Enabla^2 f_l\|_{L^{m}(\omegarhaf)} \le C \alpha.
\end{equation*}
We can assume that $f_h$ has average zero over $T_{r/2}\cap\Omega$ (otherwise we subtract the average, and add it to $f_l$, then the average of $f_l$ is the same as the average of $f$, which is bounded by $\alpha$)
so that by Poincar\'e
\begin{equation*}
 \|f_h\|_{W^{1,p}(\omegarhaf)} \le C \eps,  \quad  \| f_h\|_{W^{1,m}(\omegarhaf)}   \le C \alpha.
\end{equation*}
Finally, from $f_l=f-f_h$ we obtain
$\|f_l\|_{W^{1,m}}\le C\alpha$.
\end{proof}

\begin{lemma}\label{lemmaxntodiv}
Let $\omegarf$ be as in \eqref{eqdefomegargamma},
$T_{r,L}$ as in \eqref{eqdefCrL},
$A\in L^1(\omegarf)$. Then there is
$A^*\in L^1_\loc(\omegarf;\R^n)$
such that
the distributional divergence obeys
$\Div A^*\in L^1_\loc(\omegarf)$ and
\begin{equation}\label{eqAApApp}
 \frac{A(x)}{\gamma(x')-x_n}
 = \Div A^*(x)   .
\end{equation}
Further, if $A\in L^p(\omegarf)$ for some $p\in(1,\infty)$ then $A^*\in L^p(\omegarf;\R^n)$ and
\begin{equation}
 \|A^*\|_{L^p(\omegarf)}\le C\|A\|_{L^p(\omegarf)}
\end{equation}
with $C=C(n,p)$.
\end{lemma}
\begin{proof}
We define $A^*: \omegarf\to\R^n$ setting
\begin{equation}
  A^*(x',x_n):=
  e_n\int^{x_n}_{-2rL}
  \frac{A(x',y_n)}{\gamma(x')-y_n} dy_n.
 \end{equation}
One easily checks that
\begin{equation}
\Div A^*(x)=\partial_n A^*_n(x)=\frac{A(x',x_n)}{\gamma(x')-x_n}
\end{equation}
distributionally in $\omegarf$, which proves \eqref{eqAApApp}.

In order to check the $L^p$ estimate,
we first recall that  for any
$f\in L^p((0,a))$ the function $g(t):=\int_t^a f(s)/s ds$ is in $L^p((0,a))$, with $\|g\|_p\le p \|f\|_p$.
By density it suffices to check this for smooth functions. Integrating by parts and using $tg(t)\to 0$ for $t\to0$, $g(a)=0$ and then H\"older,
\begin{equation}
\|g\|_p^p=
 \int_0^a |g|^p(t)dt =
 t|g|^p(t)\Bigr|_{t=0}^{t=a}
 + p \int_0^a t |g|^{p-1}(t) \frac{f(t)}{t}dt
 \le p \|g\|_p^{p-1}\|f\|_p.
\end{equation}
We use this inequality for fixed $x'$, with $a=2rL+\gamma(x')$,
$f(s) = A(x', \gamma(x')-s)$, making the change of variables $y_n = \gamma(x') - s$, and obtain
\begin{equation}
 \|A^*\|_{L^p(\omegarf)}
 \le p\|A\|_{L^p(\omegarf)}.
\end{equation}
\end{proof}

\begin{proof}[Proof of Proposition~\ref{propLipschitzboundarylocal}]
We define
\begin{equation}
 \hat B:= \chi_{\{|D^2f|\le 2B\}}D^2f,\hskip5mm
 \hat A(x):=(\gamma(x')-x_n)D^2f(x) \chi_{\{|D^2f|> 2B\}}(x).
\end{equation}
Then
$|\hat B|\le 2B$ and $|\hat A|\le 2A$ pointwise, so that
$\|\hat A\|_{L^m}+\|\hat B\|_{L^m}\le 2\alpha$, $\|\hat A\|_{L^p}\le2\eps$.
Using Lemma~\ref{lemmaxntodiv} we then obtain
$\hat A^*$ such that
\begin{equation}
 D^2f = \Div \hat A^*+B^*,
\end{equation}
with the same bounds (up to constants).
The result follows then by
Proposition~\ref{propLipschitzboundary}. \end{proof}

\section{Implications of approximability in the Euclidean case}\label{sec:implEucl}
In the Euclidean case, from Theorem~\ref{th:W2m_approximation} one easily obtains
the following approximation statement.
\begin{corollary}\label{coreuclidean}
 If $U\subset\R^n$ is open, bounded, connected, Lipschitz, $p\in (1,\infty)$, $m\in (1,\infty)$, then there is $C=C(U,n,p,m)$ such that for any $f\in W^{1,p}(U;\R^n)$ there is $\hat f\in W^{2,m}(U;\R^n)$ with
 \begin{equation}\label{eqrigeuclappr1}
  \|Df-D\hat f\|_{L^p(U)} \le C \|\dist(Df, SO(n))\|_{L^p(U)}
 \end{equation}
 and 
 \begin{equation}\label{eqrigeuclappr2}
 \|D\hat f\|_{L^\infty(U)}+ \|D\hat f\|_{W^{1,m}(U)} \le C .
 \end{equation}
\end{corollary}
\begin{proof}
 The key point is to find compact manifolds $M$ and $N$ such that this fits the assumptions of Theorem~\ref{th:W2m_approximation}.

 Let $K>0$ be such that $U\subset \subset B_K$. Then $M=\overline B_K$, or, alternatively, $M$ being the torus $[-K,K)^n$, will do.

 If $p>n$, then a similar construction can work for $N$. In the general case, one starts with Lipschitz truncation, using the result corresponding to \Cref{proptruncation} in the Euclidean setting (the Euclidean statement is indeed much older than \Cref{proptruncation}, see for example \cite[Prop.~A.1 and the paragraph of the proof of Prop.~3.4]{FrieseckeJamesMueller2002-CPAM}). Let $\Lambda>0$ and $\tilde C>0$ be such that for any $f\in W^{1,p}(U;\R^n)$ there is $\tilde f$ with $\Lip(\tilde f)\le\Lambda$ and
 $\|Df-D\tilde f\|_{L^p(U)}\le \tilde C\|\dist(Df,SO(n))\|_{L^p(U)}$. We can further assume 
 $\tilde f(0)=0$. Then $\tilde f(U)\subset B_{\Lambda K}$,
 so that we can choose $N$ to be the torus
 $[-\Lambda K, \Lambda K)^n$ and obtain the assertion for $\tilde f$, the result for $f$ (with the same $\hat f$) follows by a triangular inequality. We stress that the choices of $M$ and $N$ do not depend on $f$ but only on $n$, $p$ and $U$.
\end{proof}

We show here how approximability leads to direct proofs of classical results. We focus on rigidity and linearization of nonlinear elasticity.

\subsection{Approximability implies rigidity}
\label{sectEucliRigidity}
We show how approximability  (\Cref{coreuclidean}), Korn's inequality, and Liouville's theorem suffice
to obtain the Friesecke-James-Müller nonlinear rigidity estimate from \cite[Theorem~3.1]{FrieseckeJamesMueller2002-CPAM}
(for the $L^p$ version, see \cite{ContiDolzmannMueller2014}).
\begin{proposition}
Let $U\subset\R^n$ be open, bounded, connected, Lipschitz,   $p\in (1,\infty)$. Then there is $C=C(U,n,p)$ such that for any $f\in W^{1,p}(U;\R^n)$ one has 
  \begin{equation}\label{eqthesisrigideuclidean}
   \min_{A\in \R^{n\times n}} \|Df-A\|_{L^p(U)} \le C \|\dist(Df, SO(n))\|_{L^p(U)}.
  \end{equation}
\end{proposition}
One easily checks that the matrix $A$ can be chosen to be in $SO(n)$.

It is immediate to see that~\eqref{eqthesisrigideuclidean} implies Corollary~\ref{coreuclidean}. We prove the other implication.

\begin{proof}
If \eqref{eqthesisrigideuclidean}  does not hold, there is a sequence $f_j\in W^{1,p}$ such that
\begin{equation}\label{eqeuclrigcontradass}
 \min_{A\in \R^{n\times n}} \|Df_j-A\|_{L^p} \ge j\|\dist(Df_j, SO(n))\|_{L^p}.
\end{equation}
Since $|Df_j|\le \sqrt n + \dist(Df_j, SO(n))$,
we obtain with $A=0$
\begin{equation}\label{eqdfjsonp}
\|\dist(Df_j, SO(n))\|_{L^p}\le \frac{1}{j-1} |\Omega|^{1/p} \sqrt n \to0.
\end{equation}

Fix $m>n$ and for each $j$ let $\hat f_j\in W^{2,m}(\Omega;\R^n)$ be the function from Corollary~\ref{coreuclidean}. After adding a constant and extracting a subsequence, we can assume that the sequence $\hat f_j$ has a weak limit  $\hat f$ in $W^{2,m}$. Then $D\hat f_j\to D\hat f$ uniformly and, using \eqref{eqrigeuclappr1} and \eqref{eqdfjsonp},
\begin{equation}\label{eqdhatfson}
 \|\dist(D\hat f, SO(n))\|_{L^p} =
  \lim_j\|\dist(D\hat f_j, SO(n))\|_{L^p} =0,
\end{equation}
therefore $D\hat f\in SO(n)$ pointwise, and by Liouville's rigidity theorem we obtain that $\hat f$ is affine,  $\hat f(x)=b+Qx$ for some $Q\in SO(n)$.

For each $j$, let $A_j$ be the average of $D\hat f_j$;
using that $D\hat f_j$ converges uniformly to $Q=D\hat f$ we obtain
$A_j\to Q$.
We
choose $Q_j\in SO(n)$ such that
\begin{equation}
 |A_j-Q_j|=\dist(A_j,SO(n)),
\end{equation}
obviously $Q_j\to Q$.
By minimality, 
$Q_j^TA_j$ is a symmetric matrix. Indeed,
for any skew-symmetric matrix $S$ we have that
\begin{equation}\label{eqQAsym}
\begin{split}
 0=\frac{d}{dt}|A_j-Q_j(\Id+t S)|^2 \biggr|_{t=0}=&
 \frac{d}{dt}|Q_j^TA_j-\Id-t S|^2 \biggr|_{t=0}\\
 =&
 -2S :(Q_j^TA_j-\Id).
\end{split}\end{equation}

A simple algebraic estimate shows that for any $F\in\R^{n\times n}$ one has
\begin{equation}\label{eqsonlinear}
\left| \frac{F+F^T}{2}-\Id\right|\le \dist(F,SO(n)) + 2 |F-\Id|^2\,.
\end{equation}
We define $g_j\in W^{1,p}(\Rndomain;\R^n)$ by $\hat f_j(x)=Q_j(x+g_j(x))$,
apply \eqref{eqsonlinear} to $Q_j^TD\hat f_j=\Id+Dg_j$ and obtain
\begin{equation}
\left| \frac{D g_j+ Dg_j^T}{2}\right|\le \dist(D\hat f_j,SO(n)) + 2 |Dg_j|^2\,.
\end{equation}
Further, the average of $Dg_j$ equals $Q_j^TA_j-\Id$, which is a symmetric matrix by \eqref{eqQAsym}.
By Korn's inequality,
\begin{equation}
 \|Dg_j\|_{L^p} \le C_K \|\dist(D\hat f_j,SO(n))\|_{L^p} + 2C_K \|Dg_j\|_{L^\infty}
 \|Dg_j\|_{L^p}\,.
\end{equation}
By uniform convergence of $Q_j^TD\hat f_j\to \Id$, for $j$ sufficiently large we have $2C_K\|Dg_j\|_{L^\infty}\le 1/2$, so that
\begin{equation}
\|D\hat f_j-Q_j\|_{L^p}=
 \|Dg_j\|_{L^p} \le 2C_K \|\dist(D\hat f_j,SO(n))\|_{L^p}.
\end{equation}
By \eqref{eqrigeuclappr1} the estimate (with a different constant) holds for $f_j$, and this contradicts \eqref{eqeuclrigcontradass}.
\end{proof}
\subsection{Linearization of nonlinear elasticity}
\label{sec:eucllinearization}
We show how approximability  (\Cref{coreuclidean}), Korn's inequality, and Liouville's theorem similarly suffice to prove directly the compactness and lower bound properties in the linearization of nonlinear elasticity, see \cite{DalmasoNegriPercivale2002}.
Let
\begin{equation}
 E_\eps(f):=\frac1{\eps^2} \int_\Rndomain W(Df) dx,
\end{equation}
with $W(\Id)=0$, $W(F)\ge C\dist^2(F,SO(n))$, $W(QF)=W(F)$ for $Q\in SO(n)$, $W$ twice differentiable in a neighbourhood of the identity.
\begin{proposition}
Let $U\subset\R^n$ be open, bounded, connected, Lipschitz.
Assume $f_\eps\in W^{1,2}(\Rndomain;\R^n)$ is given with $\limsup_\eps E_\eps(f_\eps)<\infty$. Then, up to a subsequence, there are $Q_\eps\in SO(n)$ and $b_\eps\in\R^n$ such that
 the displacements $u_\eps(x):= (Q_\eps^T(f_\eps(x)-b_\eps)- x)/\eps$ have a weak limit $u$ in $W^{1,2}$ and,
 letting $\C:=D^2W(\Id)$,
 \begin{equation}
  \int_\Rndomain \frac12 \C Du \cdot Du dx \le \liminf_\eps E_\eps(f_\eps).
 \end{equation}
\end{proposition}
\begin{proof}
We can assume $f_\eps$ to have average zero.
Let $m>n$. By Corollary~\ref{coreuclidean} there are $\hat f_\eps$ such that 
\begin{equation}
\|D^2\hat f_\eps\|_{L^m}\le C \hskip3mm\text{ and }\hskip3mm
\|D\hat f_\eps-Df_\eps\|_{L^2}\le C
 \| \dist(Df_\eps, SO(n))\|_{L^2} \le C\eps.
\end{equation}
By Ascoli-Arzel\'a  there is a subsequence with $\hat f_\eps\to \hat f$ in $C^1$, with $D\hat f\in SO(n)$, so that by Liouville's rigidity theorem $\hat f(x)=Qx+b$ for some $Q\in SO(n)$. We select $Q_\eps\in SO(n)$, $Q_\eps\to Q$, so that
\begin{equation}\label{eqavgsym}
 Q_\eps^T \int_\Rndomain D\hat f_\eps dx \in \R^{n\times n}_\sym
\end{equation}
and $b_\eps\in\R^n$ to be the average of $\hat f_\eps(x)-x$.

Fix $\delta>0$, chosen below. For $\eps$ sufficiently small, we have $\|Q_\eps^T D\hat f_\eps-\Id\|_{L^\infty}\le\delta$. We use the Taylor formula
\begin{equation}\label{eqWdhatfepslb}
W(D\hat f_\eps)=W(Q_\eps^TD\hat f_\eps)\ge \frac12 \C (Q_\eps^TD\hat f_\eps-\Id) \cdot (Q_\eps^TD\hat f_\eps-\Id)
-\omega(\delta) |Q_\eps^TD\hat f_\eps-\Id|^2.
\end{equation}
By Korn's inequality and~\eqref{eqavgsym},
\begin{equation}
 \int \frac12 \C (Q_\eps^TD\hat f_\eps-\Id) \cdot (Q_\eps^TD\hat f_\eps-\Id)dx
 \ge C_K \int|Q_\eps^TD\hat f_\eps-\Id|^2dx.
\end{equation}
Choosing $\delta$ so that $\omega(\delta)\le \frac12 C_K$
and integrating \eqref{eqWdhatfepslb},
\begin{equation}
  \int|Q_\eps^TD\hat f_\eps-\Id|^2dx\le C \int W(D\hat f_\eps)dx\le C \int W(Df_\eps)dx\le C\eps^2, 
\end{equation}
and, letting $u_\eps(x):=Q_\eps^T(f_\eps(x)-x-b_\eps)/\eps$ and using $\|D\hat f_\eps-Df_\eps\|_{L^2}\le C\eps$,
\begin{equation}
\int |Du_\eps|^2dx \le \frac1{\eps^2}  \int|Q_\eps^TDf_\eps-\Id|^2dx\le C, 
\end{equation}
which concludes the compactness proof.

The lower bound then follows with the usual strategy: one writes
\begin{equation}
\frac1{\eps^2}W(Df_\eps)\ge \frac12 \C Du_\eps \cdot Du_\eps
-\omega(|Df_\eps-Q_\eps|) |Du_\eps|^2
\end{equation}
lets $F_\eps:=\{|Df_\eps-Q_\eps|>\delta\}$, observes that 
$|F_\eps|\to0$, so that $\chi_{F_\eps} Du_\eps\weakto0$ in $L^2$,
and
\begin{equation}\begin{split}
 \liminf_\eps E_\eps(f_\eps) \ge &\liminf_\eps 
\frac12 \int_\Rndomain \chi_{\Rndomain\setminus F_\eps} \C Du_\eps \cdot Du_\eps - \omega(\delta) |Du_\eps|^2 dx\\
 \ge& \frac12\int_\Rndomain \C Du \cdot Du dx - \omega(\delta) \limsup_\eps \|Du_\eps\|_2^2.
\end{split}\end{equation}
Since $Du_\eps$ is bounded in $L^2$, and $\delta$ is arbitrary, the proof is concluded.
\end{proof}

\section*{Acknowledgments}
This work was partially funded by the Deutsche Forschungsgemeinschaft (DFG, German Research Foundation) {\sl via}
 project 539309657, CRC 1720, project 
441468770, SPP 2256
 and project 390685813, GZ 2047/1.

\addcontentsline{toc}{section}{References}
\def\cprime{$'$}

\end{document}